\documentclass[10pt,fleqn]{article}

\usepackage{color}

\usepackage{mathptmx}
\usepackage[T1]{fontenc}

\usepackage{latexsym,epsfig,verbatim}
\usepackage{amsmath,amsthm,amssymb}
\usepackage{dsfont} 
\input{xy}
\xyoption{all}

\usepackage{url}
\usepackage{color}
\usepackage[colorlinks,citecolor=blue]{hyperref}

\usepackage{accents}

\usepackage{tikz}
\usetikzlibrary{matrix,arrows,decorations.pathmorphing}

\numberwithin{equation}{section} 

\theoremstyle{plain}

\newtheorem{theorem}{Theorem}
\newtheorem*{width}{Theorem \ref{width.theorem}}
\newtheorem*{generalbundles}{Theorem \ref{T:generalbundlewidth}}

\newtheorem{lemma}[theorem]{Lemma}
\newtheorem{corollary}[theorem]{Corollary}

\newtheorem{proposition}[theorem]{Proposition}

\newtheorem{claim}[theorem]{Claim}
\newtheorem*{claimplain}{Claim}
\newtheorem{criterion}[theorem]{Criterion}

\theoremstyle{definition}

\newtheorem*{remarks}{Remarks}

\newcommand{\HH}{\mathbb{H}}

\newcommand{\RR}{\mathbb{R}}
\newcommand{\ZZ}{\mathbb{Z}}

\newcommand{\calA}{\mathcal{A}}

\newcommand{\calC}{\mathcal{C}}
\newcommand{\calD}{\mathcal{D}}
\newcommand{\calE}{\mathcal{E}}
\newcommand{\calF}{\mathcal{F}}
\newcommand{\calG}{\mathcal{G}}
\newcommand{\calH}{\mathcal{H}}

\newcommand{\calK}{\mathcal{K}}

\newcommand{\calN}{\mathcal{N}}
\newcommand{\calP}{\mathcal{P}}
\newcommand{\calQ}{\mathcal{Q}}
\newcommand{\calR}{\mathcal{R}}
\newcommand{\calS}{\mathcal{S}}

\newcommand{\calU}{\mathcal{U}}
\newcommand{\calV}{\mathcal{V}}
\newcommand{\calX}{\mathcal{X}}
\newcommand{\MitraQIK}{{K_{1}}}
\newcommand{\MitraQIC}{{C_{1}}}
\newcommand{\MitraProjK}{K_{0}}

\newcommand{\MosherQISectionK}{{k_{0}}}
\newcommand{\MosherQISectionC}{{c_{0}}}
\newcommand{\OurQISectionK}{{K_{2}}}
\newcommand{\OurQISectionC}{{C_{2}}}
\newcommand{\GapOverShortest}{{D_1}}
\newcommand{\GapOverRelShortest}{{D_0}}
\newcommand{\FewShortSectionsBound}{{B_1}}
\newcommand{\UniformPropernessConstantA}{{E_0}}

\newcommand{\HypFlatQIK}{{K_3}}
\newcommand{\HypFlatQIC}{{C_3}}
\newcommand{\HypFlatQIKgen}{{K_6}}
\newcommand{\HypFlatQICgen}{{C_6}}
\newcommand{\ourepsilon}{{\epsilon_{0}}}
\newcommand{\ourtheta}{{\theta_{0}}}
\newcommand{\OurT}{{T_r}}
\newcommand{\NaughtT}{{T_0}}
\newcommand{\FirstT}{{T_1}}
\newcommand{\SecondT}{{T_2}}
\newcommand{\ThirdT}{{T_3}}
\newcommand{\SOLfiberQCXconstantA}{{A_0}}
\newcommand{\SOLfiberQIconstantK}{{K_4}}
\newcommand{\SOLfiberQIconstantC}{{C_4}}
\newcommand{\BoundOnSectionsInNeighborhood}{{B_{2}}}
\newcommand{\GapOverBalance}{{D_3}}
\newcommand{\GapOverRelBalance}{{D_2}}

\newcommand{\FiberwiseQIK}{{K_5}}
\newcommand{\FiberwiseQIC}{{C_5}}
\newcommand{\arccosh}{\mathrm{arccosh}}
\newcommand{\closest}[1]{\mathfrak{p}_{#1}}

\newcommand{\diam}{\mathrm{diam}}

\newcommand{\Mod}{\mathrm{Mod}}
\newcommand{\Homeo}{\mathrm{Homeo}}

\newcommand{\Sol}{\textsc{Sol}}
\newcommand{\Graph}[1]{X_{#1}}

\newcommand{\Teich}{\mathcal{T}}

\newcommand{\cusp}{\mathbf{P}}
\newcommand{\WHull}{\mathrm{WH}}
\newcommand{\isosection}{{\Xi}}
\newcommand{\co}{\colon\thinspace}

\newcommand{\scriptinfinity}{{\rotatebox{90}{\scriptsize $8$}}}

\begin{document}

\title{\textbf{A geometric criterion to be pseudo-Anosov}}
\author{Richard P. Kent IV and Christopher J. Leininger\thanks{The first author was supported in part by an NSF MSPRF and NSF grant DMS-1104871, the second author by NSF grants DMS-0603881 and DMS-0905748.
Both authors were partially supported by the GEAR network.
}}

\date{April 5, 2014}

\maketitle

\begin{abstract} 
If $S$ is a hyperbolic surface and $\mathring{S}$ the surface obtained  from $S$ by removing a point, the mapping class groups $\Mod(S)$ and  $\Mod(\mathring{S})$ fit into a short exact sequence
	\[
		1 \to \pi_1(S) \to \Mod(\mathring{S}) \to \Mod(S) \to 1.
	\]
We give a new criterion for mapping classes in the kernel to be pseudo-Anosov using the geometry of hyperbolic $3$--manifolds.  
Namely, we show that if $M$ is an $\epsilon$--thick hyperbolic manifold homeomorphic to $S \times \RR$, then an element of $\pi_1(M) \cong \pi_1(S)$ represents a pseudo-Anosov element of $\Mod(\mathring{S})$ if its geodesic representative is ``wide."   
We establish similar criteria where $M$ is replaced with a coarsely hyperbolic surface bundle  coming from a $\delta$--hyperbolic surface--group extension.
\end{abstract} 

\section{Introduction: mapping classes from fibrations}
If $X$ is a surface, let $\Mod(X) = \pi_0(\Homeo^+(X))$ be its mapping class group and let $\mathring{X}$ be the surface obtained from $X$ by removing a point. 

Surface bundles $X \to E \to B$ over a space $B$ with fiber $X$ are determined by homomorphisms $\pi_1(B) \to \Mod(X)$; see \cite{Morita.book}.   Thurston's Geometrization Theorem for fibered $3$--manifolds opens the door to an investigation of the geometric behavior of such surface bundles.   For instance, there are necessary and sufficient geometric conditions on $\pi_1(B) \to \Mod(X)$ that guarantee that $\pi_1(E)$ is word-hyperbolic; see \cite{Farb.Mosher.2002,hamenstadt}.  To verify these conditions, one is often faced with the problem of determining when a subgroup $G < \Mod(X)$ is purely pseudo-Anosov, a problem we take up here.

To describe our first result, let $N$ be a closed hyperbolic $3$--manifold that fibers over the circle with fiber a surface $S$, and let $N_\ZZ \to N$ be the corresponding infinite cyclic covering of $N$.
The long exact sequence of the fibration is concentrated in a short exact sequence 
	\begin{equation}\label{FibrationSequence}
		\begin{tikzpicture}[>= to, line width = .075em, baseline=(current bounding box.center)]
		\matrix (m) [matrix of math nodes, column sep=1.5em, row sep = 1em, 		text height=1.5ex, text depth=0.25ex]
		{
			1 & \pi_1(S)  & \pi_1(N) & \ZZ & 1 \\
		};
		\path[->,font=\scriptsize]
		(m-1-1) edge (m-1-2)
		(m-1-2) edge (m-1-3)
		(m-1-3) edge (m-1-4)
		(m-1-4) edge (m-1-5)
		;
		\end{tikzpicture}
	\end{equation}
which injects into the Birman exact sequence \cite{Birman}
	\[
	\begin{tikzpicture}[>= to, line width = .075em, baseline=(current bounding 		box.center)]
		\matrix (m) [matrix of math nodes, column sep=1.5em, row sep = 1em, 		text height=1.5ex, text depth=0.25ex]
		{
		1 & \pi_1(S)  & \Mod(\mathring{S}) & \Mod(S) & 1. \\
		};
		\path[->,font=\scriptsize]
		(m-1-1) edge (m-1-2)
		(m-1-2) edge (m-1-3)
		(m-1-3) edge (m-1-4)
		(m-1-4) edge (m-1-5)
		;
	\end{tikzpicture}
	\]
Choosing a lift $t$ of the generator of $\ZZ$ to $\pi_1(N)$, any element of $\pi_1(N)$ may be written uniquely as a product $g t^k$, where $g$ is an element of $\pi_1(S)$.
When $k$ is nonzero, this element represents a pseudo-Anosov mapping class in $\Mod(\mathring{S})$.
When $k$ is zero, this element lies in $\pi_1(S)$, and, by a theorem of Kra \cite{Kra} (see also \cite{KentLeiningerSchleimer}), it is pseudo-Anosov in $\Mod(\mathring{S})$ if and only if it fills $S$.
These observations were first made by Ian Agol \cite{Agol.PrivateCommunication}.

\begin{criterion}[Agol's criterion]\label{AgolCriterion} A subgroup $H$ of $\pi_1(N)$ is a purely pseudo-Anosov subgroup of $\Mod(\mathring{S})$ if and only if every nontrivial element of $H \cap \pi_1(S)$ fills $S$.
\end{criterion}

This topological criterion is very difficult to check.
Our main theorem is a geometric criterion for an element of $\pi_1(N_\ZZ)$ to be filling.

\begin{width}  Let $S$ be a closed oriented surface of Euler characteristic $\chi = \chi(S) < 0$ and let $\epsilon$ and $K$ be positive numbers.  
There is a $W = W(\chi, \epsilon, K) > 0$ such that the following holds. 
Equip $M = S \times \RR$ with any $\epsilon$--thick hyperbolic structure, and let $\ell \co M \to \RR$ be a $K$--Lipschitz submersion. 
If $Y$ is a proper incompressible subsurface of $S$ and $\calC_Y$ is the convex core of the corresponding cover of $M$, then the width $\diam(\ell(\calC_Y))$ of $\calC_Y$ is at most  $W$.
In particular, if $\gamma$ is a geodesic loop in $M$ such that $\diam(\ell(\gamma)) > W$, then $\gamma$ fills $S$.  
\end{width}

If $\diam(\ell(\gamma)) > W$, we say that $\gamma$ is \textit{wide}.
Agol's criterion then becomes:

\begin{criterion}[Width criterion]\label{WidthCriterion} A subgroup $H$ of $\pi_1(N)$ is a purely pseudo-Anosov subgroup of $\Mod(\mathring{S})$ if every nontrivial element of $H \cap \pi_1(S)$ is wide.
\end{criterion}
\noindent
\begin{remarks}
1. The fact that geodesic representatives in $N$ of elements of $\pi_1(S)$ realized by {\it simple} closed curves on $S$ are not wide is fairly straightforward.\\
2. Filling elements need not be wide.
\end{remarks}

This criterion, and Theorem \ref{width.theorem}, arose out of the authors' attempts to find purely pseudo-Anosov surface subgroups of mapping class groups by exploiting the abundance of surface subgroups of hyperbolic $3$--manifold groups (see \cite{kahnmarkovic}).   

In Section \ref{S:CuspedCase} we prove a generalization of Theorem \ref{width.theorem} to the case of punctured surfaces, Theorem \ref{width.cusps.theorem}.  The authors and S. Dowdall use these theorems to prove the following.
\begin{theorem} [Dowdall--Kent--Leininger \cite{DowdallKentLeininger}]
Suppose $N$ is a finite volume hyperbolic $3$--manifold that fibers over the circle with fiber $S$ and $G < \pi_1(N)$.   
As a subgroup of $\Mod(\mathring{S})$, $G$ is convex cocompact in the sense of Farb and Mosher \cite{Farb.Mosher.2002} if and only if $G$ is finitely generated and purely pseudo-Anosov.
\end{theorem}
In particular, this answers a special case of Question 1.5 of \cite{Farb.Mosher.2002}, and generalizes Theorem 6.1 of \cite{KentLeiningerSchleimer}.

\bigskip
\noindent
In Section \ref{S:GeneralBundles}, we generalize Theorem \ref{width.theorem} in a different direction by replacing $M$ with a hyperbolic surface--group extension $\Gamma$.
\begin{generalbundles} 
Let 
	\begin{equation}\label{EQ:ses}
	\begin{tikzpicture}[>= to, line width = .075em, baseline=(current bounding 		box.center)]
		\matrix (m) [matrix of math nodes, column sep=1.5em, row sep = 1em, 		text height=1.5ex, text depth=0.25ex]
		{
		1 & \pi_1(S)  & \Gamma & G & 1 \\
		};
		\path[->,font=\scriptsize]
		(m-1-1) edge 					(m-1-2)
		(m-1-2) edge 					(m-1-3)
		(m-1-3) edge 	node[auto]{$\ell$}	(m-1-4)
		(m-1-4) edge 					(m-1-5)
		;
	\end{tikzpicture}
	\end{equation}
be a short exact sequence with $\Gamma$ a hyperbolic group, and equip $\Gamma$ and $G$ with word metrics on finite generating sets.  
There is a $W > 0$ such that, given any nonfilling $\gamma$ in $\pi_1(S)$ and any $\gamma$--quasiinvariant geodesic $\mathcal{G}$ in $\Gamma$, we have $\diam(\ell(\mathcal{G})) \leq W$.
\end{generalbundles}
\noindent Given an infinite cyclic subgroup of $G$, one obtains a short exact sequence
\[
	\begin{tikzpicture}[>= to, line width = .075em, baseline=(current bounding 		box.center)]
		\matrix (m) [matrix of math nodes, column sep=1.5em, row sep = 1em, 		text height=1.5ex, text depth=0.25ex]
		{
		1 & \pi_1(S)  & \Gamma_\ZZ & \ZZ & 1 \\
		};
		\path[->,font=\scriptsize]
		(m-1-1) edge 					(m-1-2)
		(m-1-2) edge 					(m-1-3)
		(m-1-3) edge 					(m-1-4)
		(m-1-4) edge 					(m-1-5)
		;
	\end{tikzpicture}
\]
that injects into \eqref{EQ:ses}, and one may be tempted to argue that Theorem \ref{T:generalbundlewidth} thus follows quickly from Criterion \ref{WidthCriterion}.
This attack is thwarted by the fact that $\Gamma_\ZZ$ is wildly metrically distorted in $\Gamma$.

Again, the authors and S. Dowdall apply Theorem \ref{T:generalbundlewidth} to prove the following theorem.
\begin{theorem}[Dowdall--Kent--Leininger \cite{DowdallKentLeininger}]\label{ConvexCocompactSubgroups.theorem}
Let
	\[
	\begin{tikzpicture}[>= to, line width = .075em, baseline=(current bounding 		box.center)]
		\matrix (m) [matrix of math nodes, column sep=1.5em, row sep = 1em, 		text height=1.5ex, text depth=0.25ex]
		{
		1 & \pi_1(S)  & \Gamma & G & 1 \\
		};
		\path[->,font=\scriptsize]
		(m-1-1) edge 					(m-1-2)
		(m-1-2) edge 					(m-1-3)
		(m-1-3) edge 					(m-1-4)
		(m-1-4) edge 					(m-1-5)
		;
	\end{tikzpicture}
	\]
be a short exact sequence with $\Gamma$ hyperbolic.
Any quasiconvex finitely generated purely pseudo-Anosov subgroup of $\Gamma \subset \Mod(\mathring S)$ is convex cocompact. \qed
\end{theorem}

\bigskip
\noindent
\textbf{Acknowledgments.} The authors thank
Ian Agol, 
Jeff Brock, 
Dick Canary
and
Yair Minsky
for helpful conversations.
The authors also thank the referees for suggestions that have improved the paper considerably.

\section{Criterion to fill}
If $M$ is manifold, $\ell \co M \to \RR$ is a function, and $X$ is a subset of $M$, we define the \textit{width of $X$ with respect to $\ell$} (or simply the \textit{width of $X$}) to be $\diam(\ell(X))$.
If $X$ is a subset of any covering space of $\Pi \co N \to M$, we define the \textit{width} of $X$ to be $\diam(\ell(\Pi(X)))$.

Let $S$ be a closed orientable hyperbolic surface.
A closed curve in $S \times \RR$ is \textit{filling} if its projection to $S$ is filling.

If $M = S \times \RR$ is equipped with a hyperbolic metric and $Y$ is an incompressible subsurface of $S$, we let $\Gamma_Y$ be the Kleinian group corresponding to $\pi_1(Y) \subset \pi_1(M)$, and $\Pi:M_Y = \HH^3/\Gamma_Y \to M$ the corresponding cover.
We let $\calC_Y \subset M_Y$ denote the convex core.  We say that this hyperbolic structure is $\epsilon$--thick if the injectivity radius at every point is bounded below by $\epsilon$.

\begin{theorem}\label{width.theorem}  Let $S$ be a closed oriented surface of Euler characteristic $\chi = \chi(S) < 0$ and let $\epsilon$ and $K$ be positive numbers.  
There is a $W = W(\chi, \epsilon, K) > 0$ such that the following holds. 
Equip $M = S \times \RR$ with any $\epsilon$--thick hyperbolic structure, and let $\ell \co M \to \RR$ be a $K$--Lipschitz submersion. 
If $Y$ is a proper incompressible subsurface of $S$, then the width of $\calC_Y$ is at most  $W$.
In particular, if $\gamma$ is a geodesic loop in $M$ such that $\diam(\ell(\gamma)) > W$, then $\gamma$ fills $S$.  
\end{theorem}

When $M$ is the cover of a fibered hyperbolic $3$--manifold corresponding to the fiber, the following lemma follows from the main theorem of \cite{Scott.Swarup.1990}.

\begin{lemma}\label{SubgroupsAreSchottky} If $M = S \times \RR$ is equipped with a hyperbolic structure without parabolics, and $Y$ is a proper incompressible subsurface of $S$, then the group $\Gamma_Y$ is a Schottky group (a convex cocompact free Kleinian group).
\end{lemma}
\begin{proof}[Proof of Lemma \ref{SubgroupsAreSchottky}]
Suppose that $\Gamma_Y$ is not Schottky a group.

Since $M$ has no cusps, and $\Pi$ is a covering, $M_Y$ also has no cusps.
So $\Gamma_Y$ must be geometrically infinite.

If we let $S_Y$ denote the covering of $S$ corresponding to $Y$ (which is homeomorphic to the interior of $Y$), then $M_Y \cong S_Y \times \RR$ is homeomorphic to the interior of a handlebody.
By Canary's Covering Theorem \cite{Canary.CoveringTheorem}, there is a neighborhood $\calE$ of the end of $M_Y$ such that $\Pi|_\calE$ is finite-to-one.
Since $\Pi$ is a covering map and $M_Y - \calE$ is compact, we conclude that $\Pi$ is finite-to-one.
But $M$ is homotopy equivalent to a closed surface and $\Gamma_Y$ is free. 
\end{proof}

\begin{proof}[Proof of Theorem \ref{width.theorem}]

Let $\partial Y^*$ be the geodesic representative of $\partial Y$ in $M$.

The geodesic multicurve $\partial Y^*$ is realized by a pleated surface $\calF \to M$ (see Theorem 5.3.6 of \cite{CanaryEpsteinGreen}).
Since $M$ is $\epsilon$--thick and $\calF \to M$ is a $1$--Lipschitz incompressible map, there is a number $B = B(\chi,\epsilon)$ that bounds the diameter of (the image of) $\calF$ in $M$.
Since $\ell$ is $K$--Lipschitz the width of $\cal F$ is at most $KB$, and hence so is the width of $\partial Y^*$. 

If $\calC_Y$ has no interior, we let $\partial \calC_Y$ be the double $\mathfrak{D} \calC_Y$, considered as a map $\mathfrak{D} \calC_Y \to \calC_Y \to M$.
Note that since $\Gamma_Y$ is Schottky, $\partial \calC_Y$ is a nonempty, compact pleated surface.

\begin{lemma}\label{NarrowBoundaries} There is a number $W = W(\chi,\epsilon, K)$ such that $\partial \calC_Y$ has width less than $W$.
\end{lemma}
\begin{proof}
Let $\delta$ be less than the minimum of $\epsilon$ and the $2$--dimensional Margulis constant.

There is a number $D = D(\chi,\epsilon)$ such that $\partial \calC_Y$ lies in the $D$--neighborhood of $\partial Y^*$.
To see this, let $\calP(\delta)$ be the $\delta$--thin part of $\partial \calC_Y$, and note that the components of $\partial \calC_Y - \calP(\delta)$ have diameters bounded above by a constant $E = E(\chi,\delta)$.
Since $M$ is $\epsilon$--thick, every loop in $\calP(\delta)$  bounds a disk in $M_Y$.
Moreover, every point in $\calP(\delta)$ lies in a loop of length less than $\delta$.
Such a loop bounds a disk in $M_Y$ of diameter at most $\delta$, and
since $\partial Y^*$ is disk--busting, every point of $\calP(\delta)$ is within $\delta$ of $\partial Y^*$. 
But every point of $\partial \calC_Y - \calP(\delta)$ is within $E$ of $\calP(\delta)$.  
Letting $D = E+\delta$, we have $\partial \calC_Y$ contained in the $D$--neighborhood of $\partial Y^*$.

Since $\partial Y^*$ has width at most $KB$, the width of $\partial \calC_Y$ is at most $W=KB + 2KD$.
\end{proof}


If $\partial \calC_Y = \calC_Y$, we are done by Lemma \ref{NarrowBoundaries}.
So we assume that $\calC_Y^\circ \neq \emptyset$.  The map $\pi \co \calC_Y \to M$ is an immersion on $\calC_Y^\circ$, and since $\ell$ is a submersion, the composition $\ell \circ \pi \co \calC_Y \to \RR$ is a submersion on $\calC_Y^\circ$ as well.
It follows that $\ell \circ \pi$ achieves its extrema on $\partial \calC_Y$.  
So the width of $\calC_Y$ equals the width of $\partial \calC_Y$, which is bounded by Lemma \ref{NarrowBoundaries}.
\end{proof}

\section{The cusped case}\label{S:CuspedCase}

Let $S$ be a noncompact finite--volume hyperbolic surface with Euler characteristic $\chi < 0$, and let $M$ be a hyperbolic manifold homeomorphic to $S \times \RR$.
Note that when $M$ is the infinite cyclic cover of a $3$--manifold fibering over the circle, the lift of the bundle projection is not a Lipschitz map to $\RR$.
As such projections are natural for measuring width, we find the naive analog of Theorem \ref{width.theorem} too restrictive.
In this section, we discuss the correct analog, where one must first project onto the complement of a neighborhood of the cusps before taking a Lipschitz projection to $\RR$ to compute widths.

Let $M = S \times \RR$, and equip $M$ with a type--preserving hyperbolic structure without accidental parabolics.  
Let $P \subset S$ denote a standard cusp neighborhood of the ends, so that $S^0 = S - P$ is a compact surface with boundary and $S^0 \to S$ is a homotopy equivalence.  
Let $\cusp = P \times \RR \subset M$ and set
	\[ 
		M^0 = M - \cusp = S^0 \times \RR.
	\]
We assume that the restriction of the hyperbolic metric to each component of $\cusp$ is isometric to a standard cusp neighborhood
	\[
		\cusp_3(r) = \big\{ (z,t) \in \HH^3 \mid t > r \big\} / \big\langle (z,t) \mapsto (z+1,t) \big\rangle,
	\]
for some $r$ satisfying $\arccosh(1 + 1/2r^2) < \mu_3$, where $\mu_3$ is the $3$--dimensional Margulis constant. 
We often write $\cusp(r) = \cusp$ when $r$ is relevant.

Given an essential subsurface $Y \subset S$, let $M_Y \to M$ denote the cover corresponding to $Y$ and $\calC_Y \subset M_Y$ its convex core.  
An argument similar to the proof of Lemma \ref{SubgroupsAreSchottky} shows that the Kleinian group $\Gamma_Y$ corresponding to $Y$ is geometrically finite without accidental parabolics.
The boundary $\partial \calC_Y$ is a locally convex pleated surface whose cusps are carried to  cusps of $M_Y$  (consequently, $\calC_Y$ is bent along a compact geodesic lamination).  
Each cusp of $\partial \calC_Y$ has a \textit{standard neighborhood} $\calU_r$ isometric to
	\[
		\cusp_2(r) = \big\{ (x,t) \in \HH^2 \mid t > r \big\} / \big\langle (x,t)
		\mapsto (x+1,t) \big\rangle 
		.
	\]
	
Note that there is a definite cusp neighborhood in any hyperbolic surface that misses every compact geodesic lamination.
To see this, fix a cusp neighborhood and consider a sequence of leaves of compact laminations reaching deeper and deeper into the cusp neighborhood.
By compactness, these leaves must be tangent to horocycles deeper and deeper in the cusp neighborhood.
But these horocycles are getting shorter and shorter, from which it is apparent that the leaves must eventually have self--intersections, providing a contradiction.
It follows that there is an $r_0 = r_0(\chi)$ such that $\calU_r$ is disjoint from the pleating locus when $r \geq r_0$. 
It follows that, for $r \geq r_0$, our $\calU_r$ is totally geodesic.
We take $r \geq \max\{ r_0 , (2\cosh(\mu_3) -1)^{-1/2}\}$, thus ensuring that $\calU_r$ is totally geodesic and carried into $\cusp$.

\begin{proposition}\label{ConvexCoreNormalForm.proposition} There is an $r = r(\chi)$ with the following property.
Equip $M = S \times \RR$ with a type--preserving hyperbolic metric without accidental parabolics,
and suppose each component of $\cusp$ is isometric to $\cusp_3(r)$. 
Let $Y \subset S$ be an essential subsurface whose corresponding cover $M_Y \to M$ has convex core $\calC_Y$.  
Then each component of the intersection of $\calC_Y$ and $\cusp$ is isometric to
	\[ 
		\cusp_3(r,R) = \left\{ (z,t) \in \HH^3 \mid t > r 
		\ \, \mathrm{and} \ \, 0 \leq \mathrm{Im}(z) \leq R \right\}/\langle (z,t) 
		\mapsto (z + 1,t) \rangle 
	\]
for some $R > 0$.
\end{proposition}
\begin{proof}
An area argument shows that if $r > 0$ is sufficiently large (depending only on $\chi$), any pleated surface representative of $S$ meets $\cusp(r)$ only in its cusps.
(To see this, note that if a pleated surface representative of $S$ plunges deep into $\cusp(r)$, its diameter would be large.
This forces one of two alternatives: either an essential curve on $S$  lies in $\cusp(r)$, violating our accidental parabolics hypothesis; or the pleated surface contains a large diameter disk, violating the Gauss--Bonnet Theorem.)
We assume that $r$ is at least this large, in addition to the constraints already imposed on $r$. 

Let $Y$ be an essential subsurface of $S$.
For a given $r > 0$, let $\calV_r$ be the union of the cusp neighborhoods $\calU_r \subset \partial \calC_Y$ constructed above.
If $r > 0$ is sufficiently large, and a point of $\partial \calC_Y - \calV_r$ is sufficiently deep in $\cusp(r)$, then area considerations again imply that $\partial \calC_Y - \calV_r$ must contain a compressible curve bounding a disk $\calD$ contained in $\calC_Y$ and some component of $\cusp(r)$.  
(As in the area argument above, the surface $\partial \calC_Y - \calV_r$ has bounded area and, paired with the no accidental parabolics hypothesis, this guarantees that any essential curve in $\partial \calC_Y - \calV_r$ lying in $\cusp(r)$ must be nullhomotopic there. This produces the desired disk.)
Since $\partial Y$ is disk--busting in $\calC_Y$, its geodesic representative $\partial Y^* \subset \calC_Y$ must intersect $\calD$, and hence $\cusp(r)$.
But this means that if $\calF \to M$ is any pleated surface representative of $S$ realizing $\partial Y^*$, then the noncuspidal part of $\calF$ must hit $\cusp(r)$, contradicting our choice of $r$.
We find that $\partial \calC_Y - \calV_r$ is carried a uniformly bounded distance (depending only on $\chi$) into $\cusp(r)$.
Choosing a larger $r$, we assume that $\partial \calC_Y$ hits $\cusp(r)$ only in the $\calU_r$.

Let $\cusp_Y(r)$ be the preimage of $\cusp(r)$ in $M_Y$.
Suppose $\calK$ is a component of $\calC_Y \cap \cusp_Y(r)$ which is not of the form $\cusp_3(r,R)$ for any $R >0$.
Then the closure of $\calK$ must intersect $\partial \cusp_Y(r)$ in a locally convex (horospherical) surface $\calH$.
This surface lies in $\calC_Y^\circ$, since $\partial \calC_Y$ hits $\cusp_Y(r)$ only in the $\calU_r$.  
Moreover, $\calH$ is compact, as $\calC_Y$ is compact after its cuspidal thin--part is thrown away.
But this all implies that $\partial  \cusp_Y(r)$ in $M_Y$ has a compact component, namely $\calH$, which is absurd.
We conclude that every component of $\calC_Y \cap \cusp_Y(r)$ has the form $\cusp_3(r,R)$.
It follows that every component of $\calC_Y \cap \cusp(r)$ has this form.
\end{proof}

We say that a hyperbolic structure on a noncompact manifold $M$ is \textbf{$\epsilon$--thick} if the length of its shortest geodesic loop is at least $\epsilon$.

\begin{theorem}\label{width.cusps.theorem}  Let $S$ be a finite--type noncompact oriented surface of Euler characteristic $\chi<0$.
Let $\epsilon$ and $K$ be positive numbers.
Equip $M = S \times \RR$ with an $\epsilon$--thick hyperbolic metric, and let $r = r(\chi)$ be the number  given by Proposition \ref{ConvexCoreNormalForm.proposition}.
There is a $W = W(\chi,\epsilon,K) > 0$  such that the following holds. 
Let $\ell \co M - \cusp(r) \longrightarrow \RR$ be a $K$--Lipschitz map and let $\nu \co M \longrightarrow M - \cusp(r)$ be the normal projection. 
If $Y$ is a proper incompressible subsurface of $S$ with convex core $\calC_Y$ mapping to $M$ via $\Pi\co \calC_Y \to M$, then $\diam(\ell(\nu(\Pi(\calC_Y)))) \leq W$.
If $\gamma$ is a geodesic loop in $M$ with $\diam(\ell(\nu(\gamma))) > W$, then $\gamma$ fills $S$.  
\end{theorem}

We define the \textit{width} of a subset $X \subset M$ to be $\diam(\ell(\nu(X)))$, and of a subset $X \subset N$ of a covering space $\Pi:N \to M$ to be $\diam(\ell(\nu(\Pi(X))))$.

\begin{lemma}\label{NarrowBoundariesWithCusps} There is a constant $W = W(\chi,\epsilon, K)$ such that $\partial \calC_Y$ has width less than $W$.
In particular, the boundary $\partial \calX_Y$ of $\calX_Y = \calC_Y - \cusp(r)$ has width less than $W$.
\end{lemma}
\begin{proof}
The proof is similar to the proof of Lemma \ref{NarrowBoundaries}.

Let $\delta$ be the minimum of $\epsilon$ and $\arccosh(1 + 1/2r^2) < \mu_3$.

We again let $\partial \calC_Y$ be the double $\mathfrak{D} \calC_Y$ when $\calC_Y$ is $2$--dimensional,  considered as a map $\mathfrak{D} \calC_Y \to \calC_Y \to M$.

There is a  $D = D(\chi,\epsilon)$ such that $\partial \calC_Y - \cusp(r)$ lies in the $D$--neighborhood of $\partial Y^*$.
To see this, let $\calP(\delta)$ be the $\delta$--thin part of $\partial \calC_Y$. 
Note that, by our choice of $\delta$, we have $\partial \calC_Y - \cusp(r) \subset \partial \calC_Y - \calP(\delta)$.

The components of $\partial \calC_Y - \calP(\delta)$ have diameters uniformly bounded above by a constant $E = E(\chi,\delta)$.

The thin part $\calP(\delta)$ is a union of cusp--neighborhoods and neighborhoods of short geodesics.
The cusp neighborhoods lie in $\cusp(r)$.
As before, the geodesic neighborhoods are within $\delta$ of the disk--busting $\partial Y^*$.

We conclude that $\partial \calC_Y - \cusp(r)$ is contained in the $D$--neighborhood of $\partial Y^*$ for $D = E + \delta$.

Since $\partial Y^*$ has width at most $KB$, the width of $\partial \calC_Y$, which is equal to the width of $\partial \calC_Y - \cusp(r)$, is at most $W=KB +2KD$.
\end{proof}

\begin{proof}[Proof of Theorem \ref{width.cusps.theorem}]

The proof is essentially the same as the proof of Theorem \ref{width.theorem}.   If $\calC_Y^\circ$ is empty, then $\calC_Y = \partial \calC_Y$ and the theorem follows immediately from Lemma \ref{NarrowBoundariesWithCusps}.   
When $\calC_Y^\circ \neq \emptyset$, we first observe that by the definition of $\nu$ and Proposition \ref{ConvexCoreNormalForm.proposition}
	\[
		\diam(\ell(\nu(\Pi(\calC_Y)))) = \diam(\ell(\Pi(\calC_Y - \cusp(r)))).
	\]
The composition $\ell \circ \Pi$ restricted to $\calC_Y^\circ -\cusp(r)$ is a submersion and hence on $\calC_Y-\cusp(r)$ attains its maximum and minimum values on $\partial  \calX_Y$.  By Lemma \ref{NarrowBoundariesWithCusps}, the width of $\calC_Y$ is at most $W$.
\end{proof}

\section{General surface bundles}\label{S:GeneralBundles}

We again assume that $S$ is a \textit{closed} surface.

We assume that the reader is acquainted with the basic notions in the study of hyperbolic groups at the level of Chapters III.H and III.$\Gamma$ of \cite{BridsonHaefliger}.

Consider a short exact sequence $1 \to \pi_1(S) \to \Gamma \to G \to 1$  where $\Gamma$ is hyperbolic, which we call a \textit{hyperbolic sequence}. 
We choose a finite generating set for $\Gamma$ containing one for $\pi_1(S)$, which in turn provides one for $G$, and we let $\Graph{\pi_1(S)}$, $\Graph{\Gamma}$, $\Graph{G}$ be the corresponding Cayley graphs. As $\Graph{G}$ is of primary importance, we often write $\Graph{} = \Graph{G}$. 
There are simplicial maps
	\begin{equation*}
		\begin{tikzpicture}[>= to, line width = .075em, baseline=(current bounding box.center)]
		\matrix (m) [matrix of math nodes, column sep=1.5em, row sep = 1em, 		text height=1.5ex, text depth=0.25ex]
		{
			 \Graph{\pi_1(S)}  & \Graph{\Gamma}  & \Graph{G}  \\
		};
		\path[->,font=\scriptsize]
		(m-1-1) edge 						(m-1-2)
		(m-1-2) edge 	node[above]{$\pi$}		(m-1-3) 	
		;
		\end{tikzpicture}
	\end{equation*}
which induce our short exact sequence.   
For any  $\gamma$ in $\Gamma$, we let $\widetilde \gamma^*$ denote any geodesic in $\Graph{\Gamma}$ whose endpoints are the ideal fixed points of $\gamma$.  
So $\widetilde \gamma^*$ is a $\gamma$--quasiinvariant geodesic.

\begin{theorem} \label{T:generalbundlewidth}
Given a hyperbolic sequence $1 \to \pi_1(S) \to \Gamma \to G \to 1$, there is a $W > 0$ such that, given any nonfilling $\gamma$ in $\pi_1(S)$ and any $\gamma$--quasiinvariant geodesic $\widetilde \gamma^*$, we have $\diam(\pi(\widetilde \gamma^*)) \leq W$.
\end{theorem}

The statement needed in \cite{DowdallKentLeininger} is the following, which follows easily from Theorem \ref{T:generalbundlewidth}.  Given a hyperbolic sequence $1 \to \pi_1(S) \to \Gamma \to G \to 1$ and a proper subsurface $Y \subset S$ with associated subgroup $\Gamma_Y < \Gamma$, we let $\WHull(\Gamma_Y)$ denote the union of all quasiinvariant geodesic axes of elements in $\Gamma_Y$, called the \textit{weak hull of $\WHull(\Gamma_Y)$}.

\begin{corollary}
Given a hyperbolic sequence $1 \to \pi_1(S) \to \Gamma \to G \to 1$, there is a $W' > 0$ such that, given any proper subsurface $Y \subset S$ with corresponding subgroup $\Gamma_Y < \Gamma$ we have $\diam(\pi(\WHull(\Gamma_Y))) \leq W'$.
\end{corollary}
\begin{proof}
Let $W$ be as in Theorem \ref{T:generalbundlewidth}, let $\delta$ be the hyperbolicity constant for $\Gamma$, and set $W' = W + 4 \delta$.  
Given two elements $\gamma_1$ and $\gamma_2$ in $\Gamma$, let $\widetilde \gamma_1^*$ and $\widetilde \gamma_2^*$ be a pair of respective quasiinvariant geodesics.  
It suffices to show that $\diam(\pi(\widetilde \gamma_1^* \cup \widetilde \gamma_2^*)) \leq W'$, since the diameter of $\pi(\WHull(\Gamma_Y))$ is bounded by the supremum of such diameters over all pairs of quasiinvariant axes for all pairs of elements in $\Gamma_Y$.

We choose points $x_i$ in $\widetilde \gamma_i^*$ with $ \diam(\pi(x_1 \cup x_2)) = \diam(\pi(\widetilde \gamma_1^* \cup \widetilde \gamma_2^*))$.
Applying $\gamma_i$ to $\widetilde \gamma_i^*$ for $i = 1,2$, we assume that $x_1$ and $x_2$ are far from $\widetilde \gamma_2$ and $\widetilde \gamma_1$, respectively.   
There is then a third element $\gamma_3$ in $\Gamma_Y$ with a quasiinvariant geodesic $\widetilde \gamma_3^*$ that contains $x_1$ and $x_2$ in its $2\delta$--neighborhood $\calN_{2\delta}(\widetilde \gamma_3^*)$.  
Since $\gamma_3$ is in $\Gamma_Y$ and $\pi$ is $1$--Lipschitz, Theorem \ref{T:generalbundlewidth} gives us
	\[ 
		\diam(\pi(\widetilde \gamma_1^* \cup \widetilde \gamma_2^*)) = \diam(\pi(x_1 \cup x_2)) \leq \diam(\calN_{2 \delta}(\widetilde \gamma_3^*)) \leq W + 4 \delta = W'. 
		\qedhere
	\]
\end{proof}

The short exact sequence
$1 \to \pi_1(S) \to \Gamma \to G \to 1$
gives us a monodromy representation $\rho\co G \to \Mod(S)$.  
By \cite{Farb.Mosher.2002}, hyperbolicity of the sequence implies that $\rho$ has finite kernel and that $G_0 = \rho(G)$ is a \textit{convex cocompact} subgroup of $\Mod(S)$, meaning that $G_0$ has a quasiconvex orbit in Teichm\"uller space. 

The preimage of $G_0$ in $\Mod(\mathring S)$ is an extension $\Gamma_{G_0}$ of $G_0$ by $\pi_1(S)$, which is the homomorphic image of $\Gamma$, and we have the commutative diagram with exact rows
	\begin{equation*}
		\begin{tikzpicture}[>= to, line width = .075em, baseline=(current bounding box.center)]
		\matrix (m) [matrix of math nodes, column sep=1.5em, row sep = 2em, 		text height=1.5ex, text depth=0.25ex]
		{
			1  
				& \pi_1(S) 
				& \Mod(\mathring S) 
				& \Mod(S) 
				& 1
				\\
			 1 
		 		& \pi_1(S)   
				& \Gamma_{G_0}  
				& G_0
				& 1 
				\\
			 1  
			 	& \pi_1(S)  
				& \Gamma 
				& G  
				& 1
				\\
		};
		\path[->,font=\scriptsize]
		(m-1-1) edge 	(m-1-2)
		(m-1-2) edge 	(m-1-3)
		(m-1-3) edge 	(m-1-4)
		(m-1-4) edge 	(m-1-5)
		
		(m-2-1) edge 	(m-2-2)
		(m-2-2) edge 	(m-2-3)
		(m-2-3) edge 	(m-2-4)
		(m-2-4) edge 	(m-2-5)
		
		(m-3-1) edge 	(m-3-2)
		(m-3-2) edge 	(m-3-3)
		(m-3-3) edge 	(m-3-4)
		(m-3-4) edge 	(m-3-5)
		;
		\draw[double distance = .15em,font=\scriptsize]
		(m-2-2) -- 	(m-1-2)
		(m-3-2) --	(m-2-2)
		;
		\path[right hook->,font=\scriptsize]
		(m-2-3) edge 	(m-1-3)
		(m-2-4) edge	(m-1-4)
		;
		\path[->>,font=\scriptsize]
		(m-3-3) edge 	(m-2-3)
		(m-3-4) edge	(m-2-4)
		;
		\end{tikzpicture}
	\end{equation*}
The map $\Gamma \to \Gamma_{G_0}$ also has finite kernel, and is thus a quasiisometry.  
Using stability of geodesics in Gromov hyperbolic spaces (Theorem III.H.1.7 of \cite{BridsonHaefliger}), one can easily check that it suffices to prove Theorem \ref{T:generalbundlewidth} when $\rho\co G \to G_0$ is an isomorphism.
We therefore assume that $G$ is a convex cocompact subgroup of $\Mod(S)$ and that $\Gamma = \Gamma_G = \Gamma_{G_0}$.

There is a canonical $S$--bundle $\calS(S)$ over Teichm\"uller space $\Teich(S)$ in which the fiber over $[m]$ in $\Teich(S)$ is identified with $S$ endowed with the hyperbolic metric $m$.  
The universal cover of this space is a hyperbolic plane bundle $\calH(S) \to \Teich(S)$.
The Bers fibration \cite{Bers.1973} identifies $\calH(S)$ and the Teichm\"uller space $\Teich(\mathring{S})$ of $\mathring S$, and we have the commutative diagram with equivariant actions
	 \begin{equation*}
		\begin{tikzpicture}[>= to, line width = .075em, baseline=(current bounding box.center)]
		\matrix (m) [matrix of math nodes, column sep=1.5em, row sep = .1em, 		text height=1.5ex, text depth=0.25ex]
		{
			1  
				& \pi_1(S) 
				& \Mod(\mathring S) 
				& \Mod(S) 
				& 1
				\\
				& \circlearrowright
				& \circlearrowright
				& \circlearrowright
				\\
		 		& \HH^2
				& \calH(S) 
				& \Teich(S)
				\\ 
				& & & 
				\\ 
				& & & 
				\\ 
				& & & 
				\\ 
				& & & 
				\\ 
				& & & 
				\\ 
				& & & 
				\\ 
				& & & 
				\\ 
				& & & 
				\\ 
				& & & 
				\\ 
				& & & 
				\\ 
				& & & 
				\\
				& & & 
				\\
				& & & 
				\\
				& & & 
				\\
				& & & 
				\\
				& S  
				& \calS(S) 
				& \Teich(S)
				\\
		};
		\path[->,font=\scriptsize]
		(m-1-1) edge 	(m-1-2)
		(m-1-2) edge 	(m-1-3)
		(m-1-3) edge 	(m-1-4) 
		(m-1-4) edge 	(m-1-5)
		(m-3-2) edge 	(m-3-3)
		(m-3-3) edge 	(m-3-4)
		(m-19-2) edge 	(m-19-3)
		(m-19-3) edge 	(m-19-4)
		(m-3-2) edge 	(m-19-2)
		(m-3-3) edge 	(m-19-3)
		(m-3-4) edge 	(m-19-4)
		;
		\end{tikzpicture}
	\end{equation*}
 
We fix a connection on $\calS(S) \to \Teich(S)$, meaning that we choose smoothly varying direct--sum decomposition of each tangent space of $\calS(S)$ into the tangent space of the fiber and a choice of \textit{horizontal space}.

We pick a $G$--equivariant embedding $\Graph{} = \Graph{G} \to \Teich(S)$ which sends edges to geodesics, and which is therefore Lipschitz.  
We have pullback bundles 
	\begin{equation*}
		\begin{tikzpicture}[>= to, line width = .075em, baseline=(current bounding box.center)]
		\matrix (m) [matrix of math nodes, column sep=1.5em, row sep = 1.5em, 		text height=1.5ex, text depth=0.25ex]
		{
			\HH^2 
				& \calH_{\Graph{}}
				& \Graph{}
				\\
			 S 
		 		& \calS_{\Graph{}}   
				& \Graph{}
				\\
		};
		\path[->,font=\scriptsize]
		(m-1-1) edge 	(m-1-2)
		(m-1-1) edge	(m-2-1)
		
		(m-1-2) edge 	(m-1-3)
		(m-1-2) edge	(m-2-2)
		
		(m-2-1) edge	(m-2-2)
		(m-2-2) edge	(m-2-3)
		
		;
		\draw[double distance = .15em,font=\scriptsize]
		(m-1-3) --	(m-2-3)
		;
		\end{tikzpicture}
	\end{equation*}
and we call $\HH^2 \to \calH_{\Graph{}} \to \Graph{}$  an \textit{associated hyperbolic plane bundle}.
For $x$ in $\Graph{}$, we let $\calH_x$ denote the fiber of $\calH_{\Graph{}} \to \Graph{}$ over $x$.
We let $\pi$ stand for any of the maps $\calH_{\Graph{}} \to \Graph{}$, $\calS_{\Graph{}} \to \Graph{}$, and $\Graph{\Gamma} \to \Graph{}$, letting context determine which is meant.

Pulling our connection back to $\calS_{\Graph{}}$, we equip $\calS_{\Graph{}}$ with a piecewise Riemannian metric that locally splits as a product of the hyperbolic metric on the fibers and the metric lifted from $\Graph{}$.
We pull this metric back to $\calH_{\Graph{}}$.

Given two points $x$ and $y$ in $\Graph{}$ and a geodesic between them, there is a parallel transport map $\calH_x \to \calH_y$ defined by following the horizontal lines of the connection over the geodesic.
Since $G$ acts cocompactly on $\Graph{}$, there is a $\MitraProjK > 0$ so that for any two points $x$ and $y$ in $\Graph{}$, this map is $\MitraProjK^{d(x,y)}$--bilipschitz with respect to the hyperbolic metrics on the fibers.

There is a fiber--preserving $\Gamma$--equivariant quasiisometry $\Graph{\Gamma} \to \calH_{\Graph{}}$ making the following diagram commute:
	\begin{equation*}
		\begin{tikzpicture}[>= to, line width = .075em, baseline=(current bounding box.center)]
		\matrix (m) [matrix of math nodes, column sep=1.5em, row sep = 1.5em, 		text height=1.5ex, text depth=0.25ex]
		{
			\Graph{\pi_1(S)} 
				& \Graph{\Gamma}
				& \Graph{}
				\\
			 \HH^2 
		 		& \calH_{\Graph{}}   
				& \Graph{}
				\\
		};
		\path[->,font=\scriptsize]
		(m-1-1) edge 	(m-1-2)
		(m-1-1) edge	(m-2-1)
		
		(m-1-2) edge 	(m-1-3)
		(m-1-2) edge	(m-2-2)
		
		(m-2-1) edge	(m-2-2)
		(m-2-2) edge	(m-2-3)
		
		;
		\draw[double distance = .15em,font=\scriptsize]
		(m-1-3) --	(m-2-3)
		;
		\end{tikzpicture}
	\end{equation*}

Given $\gamma$ in $\pi_1(S)$, let $\calA_x(\gamma)$ denote the axis of $\gamma$ in the fiber $\calH_x$ and define a subset $\calA(\gamma)$ of $\calH_{\Graph{}}$ by
	\[ 
		\calA(\gamma) 
			= \bigcup_{x \in \Graph{}} \calA_x(\gamma).
	\]
Let $x_\gamma$ in $\Graph{}$ be a point for which the translation length of $\gamma$ on $\calA_{x_\gamma}(\gamma)$ is minimal over all $\calA_x(\gamma)$. 
We endow $\calA(\gamma)$ with the subspace metric coming from the path metric on the $1$--neighborhood $\calN_1(\calA(\gamma))$, and denote both of these metrics by $d_\gamma$.

By the stability of geodesics in hyperbolic spaces (Theorem III.H.1.7 of \cite{BridsonHaefliger}), the following theorem implies Theorem \ref{T:generalbundlewidth}. 

\begin{theorem} \label{T:generalbundlewidth3}
Given a hyperbolic sequence $1 \to \pi_1(S) \to \Gamma \to G \to 1$ with associated hyperbolic plane bundle $\HH^2 \to \calH_{\Graph{}} \to \Graph{}$, there exist $K,C > 0$ such that if $\gamma$ in $\pi_1(S)$ is a nonfilling loop in $S$, then $\calA_{x_\gamma}(\gamma)$ is a $(K,C)$--quasigeodesic in $\calH_{\Graph{}}$.
\end{theorem}
\begin{proof}[Proof that Theorem \ref{T:generalbundlewidth3} implies \ref{T:generalbundlewidth}]
The quasi-isometry $\Graph{\Gamma} \to \calH_{\Graph{}}$ sends $\widetilde \gamma^*$ to a $\gamma$--quasi-invariant uniform quasi-geodesic which is therefore uniformly close to $\calA_{x_\gamma}(\gamma)$.  Since $\calA_{x_\gamma}(\gamma)$ projects to the point $x$, and the projection to $X$ is Lipschitz, the image of $\widetilde \gamma^*$ is within some uniform distance $W$ of the image of $x$.
\end{proof}

The rest of the paper is devoted to the proof of Theorem \ref{T:generalbundlewidth3}, which is inspired by the ideas in \cite{Farb.Mosher.2002}, \cite{hamenstadt}, \cite{Mj.Sardar.2012}, \cite{MasurMinsky.1999}.  
As the argument is somewhat involved, we pause to give a detailed sketch.

\subsubsection{Outline of the rest of the paper.}
\begin{proof}[Sketch of the proof of Theorem \ref{T:generalbundlewidth3}]
The basic idea is to construct a retraction $\calH_{\Graph{}} \to \calA_{x_\gamma}(\gamma)$ that is uniformly coarsely Lipschitz. 
Being \textit{coarsely Lipschitz} means that there are $K', C'>0$ so that the distance in $\calA_{x_\gamma}(\gamma)$ between the image of any two points is at most $K'$ times their distance in $\calH_{\Graph{}}$, up to an additive error of $C'$, and \textit{uniformity} means that these constants do not depend on $\gamma$.
The existence of such a map implies that $\calA_{x_\gamma}(\gamma)$ is uniformly quasigeodesic.  

The map $\calH_{\Gamma{}} \to \calA_{x_\gamma}(\gamma)$ is a composition of two maps $\calH_{\Gamma{}} \to \calA(\gamma) \to \calA_{x_\gamma}(\gamma)$.  

The construction of the first map $\calH_{\Graph{}} \to \calA(\gamma)$, and the fact that it is uniformly coarsely Lipschitz (see Lemma \ref{L:qiladder}), is due to Mitra \cite{Mitra.1998}.
(This does not use the assumption that $\gamma$ is nonfilling.)  
This first map is defined as the fiber--wise closest--point projection: the restriction to a fiber $\calH_x$ is the closest point projection to $\calA_x(\gamma)$ with respect to the hyperbolic metric on the fiber.
The details of this step are in Section \ref{S:fiberwiseprojection}.

The second map $\calA(\gamma) \to \calA_{x_\gamma}(\gamma)$ is defined using a collection $\{ \Sigma_n \}_{n \in \ZZ}$ of sections $\Sigma_n \subset \calA(\gamma)$ of the projection $\calA(\gamma) \to \Graph{}$ introduced in Section \ref{S:qisection}. 
These sections have the following properties (see Theorem \ref{T:generalbundlewidth4}):
\begin{enumerate}
\item The section map $\Graph{} \to \Sigma_n \subset \calA_{x_\gamma}(\gamma)$ is a uniform quasiisometry.
\item For any $x$ in $\Graph{}$, the fiber $\calA_x(\gamma) \cong \RR$ intersects the set of sections in a biinfinite increasing sequence of points $\{ \Sigma_n \cap \calA_x(\gamma) \}_{n \in \ZZ}$. 
In other words, the sections intersect the fibers in order, escaping to the ends.
\item In the distinguished fiber $\calA_{x_\gamma}(\gamma)$, the distance between consecutive points of $\{\Sigma_n \cap \calA_{x_\gamma}(\gamma)\}_{n \in \ZZ}$ is constant, and
\item The distance between consecutive points of $\{\Sigma_n \cap \calA_x(\gamma)\}_{n \in \ZZ}$ is uniformly bounded below.
\end{enumerate}
The existence of sections with the first and third properties is due to Mj and Sardar \cite{Mj.Sardar.2012}.
This is based on a result of Mosher \cite{Mosher.1996} that provides uniform quasiisometrically embedded sections through any point of $\calA(\gamma)$ (see Lemma \ref{L:qisectioninladder}).  

The second and fourth properties require some new ideas, explained below, and require the hypothesis that $\gamma$ is nonfilling, unlike the first and third.  Before this explanation, we  describe the map $\calA(\gamma) \to \calA_{x_\gamma}(\gamma)$.  Each of the fibers is isometric to $\RR$ and, by the second property, the sections cut these fibers into intervals.  
The union of the intervals from $\Sigma_n$ to $\Sigma_{n+1}$ over all $x$ forms a region $\calR_n$, and the map $\calA(\gamma) \to \calA_{x_\gamma}(\gamma)$ is defined by sending this entire region to the point $\Sigma_n \cap \calA_{x_\gamma}(\gamma)$.  
Uniform properness of the fibers implies that this map is uniformly coarsely Lipschitz as required.
The detailed construction of this second map is in Section \ref{S:QIAxisFromSections}.

To establish the second and fourth properties of the sections, note that for any $x$ in $\Graph{}$, there is a uniform biinfinite quasigeodesic $g$ in $\Graph{}$ through $x$ and $x_\gamma$.  
This quasigeodesic is uniformly close to a Teichm\"uller geodesic $\tau$ in $\Teich(S)$.  Moreover, the closest point projection from $g$ to $\tau$ lifts to a fiber--preserving map between the corresponding hyperbolic plane bundles
\[ 
\calH_g \to \calH_\tau,
\]
and a result of Farb and Mosher \cite{Farb.Mosher.2002} shows that this map may be taken a uniform quasiisometry.    
To understand the sequence $\{\Sigma_n \cap \calA_x(\gamma)\}_{n \in \ZZ}$, we analyze its image in $\calH_g$.  This lies in some fiber, and is a biinfinite sequence uniformly close to the axis for $\gamma$ in that fiber.  As long as all estimates are uniform, it therefore suffices to consider a sequence of sections $\{ \Sigma_n \}_{n \in \ZZ}$ of the axis bundle $\calA_\tau(\gamma)$ over $\tau$.

The Teichm\"uller geodesic $\tau$ is defined by a quadratic differential  (see \ref{S:teichmuller geodesics}).  
It is therefore natural to replace the fiber--wise hyperbolic metric on $\calH_\tau$ with the singular  \textsc{Sol} metric (see Section \ref{S:singular sol}) for which the restriction to each fiber is the Euclidean cone metric defined by the quadratic differential (see Section \ref{S:qd and flat}).  
This is done at the expense of a uniform distortion in distances (see Lemma \ref{L:FMuniformqi}), by a result of Minsky \cite{Minsky.1994}.  
We thus reduce further to the axis bundle $\calA_\tau(\gamma)^{ \textsc{Sol}}$ for $\gamma$ with respect to the singular \textsc{Sol} metric and the attendant sections $\Sigma_n$, for which we prove properties 2 and 4.

The problem is now a technical one concerning geodesics in the Euclidean cone metrics of quadratic differentials.
We refer the general audience to Section \ref{S:teichmuller geodesics and lengths} for definitions and details, and briefly sketch the key points for the expert.

The point $x_\gamma$ is uniformly close to the \textit{balance time} for $\gamma$ along $\tau$, which we take to be $\tau(0)$, so the role of $x_\gamma$ in property 3 is taken by $\tau(0)$.   
Using arguments of Masur and Minsky \cite{MasurMinsky.1999}, we prove that any segment of $\calA_{\tau(0)}(\gamma)^{\textsc{Sol}}$ of sufficient length (depending only on $\Gamma$), must increase in length exponentially in both forward and backward time along $\tau$ after a uniformly bounded amount of time (see  Proposition \ref{P:nonfillinggrowth}).   
Taking the distance between consecutive points of the fiber to be sufficiently large, properties 2 and 4 follow.

To establish this exponential growth, we argue as follows.  
There is a simple closed curve $\alpha$ disjoint from $\gamma$, since $\gamma$ is nonfilling.  
From \cite{MasurMinsky.1999}, we know that $\alpha$ becomes \textit{mostly horizontal} and \textit{mostly vertical}, respectively, after a uniformly bounded amount of time into the future and the past, respectively, measure from time zero at the balance point.   
We prove that after further uniform steps forward and backward in time, $\gamma$ itself becomes mostly horizontal and vertical, respectively.   There cannot be too many consecutive short saddle connections (by a compactness argument), and so, in the remote future and past, exponential growth kicks in for any sufficiently long segment.  This is the last step and completes the proof.

We note that, due to certain logical dependencies, the description just given does not follow the sections below linearly.  
\end{proof}

\subsection{Fiberwise projection} \label{S:fiberwiseprojection}

The following construction is due to Mitra \cite{Mitra.1998} and is used throughout his work.
Consider the map $\closest{\gamma}\co \calH_{\Graph{}} \to \calA(\gamma)$ obtained by fiberwise closest point projection to $\calA(\gamma)$.  
That is, for $z$ in $\calH_x$, let $\closest{\gamma}(z)$ be the point on $\calA_x(\gamma)$ which is closest to $z$ with respect to the hyperbolic metric on $\calH_x$.   
The following lemma is a translation to our setting of the results in Section 3 of Mitra's paper \cite{Mitra.1998}.  
We give the proof for the reader's convenience.
\begin{lemma}[Mitra \cite{Mitra.1998}] \label{L:qiladder}
Given a hyperbolic sequence $1 \to \pi_1(S) \to \Gamma \to G \to 1$ with  associated hyperbolic plane bundle $\HH^2 \to \calH_{\Graph{}} \to \Graph{}$, there are $\MitraQIK, \MitraQIC > 0$ such that for any $\gamma$ in $\pi_1(S)$, the projection $\closest{\gamma}\co \calH_{\Graph{}} \to \calA(\gamma)$ is $(\MitraQIK, \MitraQIC)$--coarsely Lipschitz.  Consequently, $\calA(\gamma)$ is $(\MitraQIK, \MitraQIC)$--quasiisometrically embedded in $\calH_{\Graph{}}$. \qed
\end{lemma}

\begin{proof}
We begin with a few observations about the metric $d_\gamma$.  For any $0 < r < 1$ and $x$ in $\Graph{}$, consider the $r$--neighborhood of the fiber over $x$ in $\Graph{}$, $\calN_r(\calH_x) = \pi^{-1}(B(x,r))$.
Because $r < 1$, $B(x,r)$ is a tree in $\Graph{}$, and so there is a unique parallel transport to the fiber $\calH_x$ for every point in $\calN_r(\calH_x)$.  We denote this map
\[ 
	\mathfrak{f}_x\co\calN_r(\calH_x) \to \calH_x.
\]
The map $\mathfrak{f}_x$ is $\MitraProjK^r$--Lipschitz and $\MitraProjK^r$--biLipschitz when restricted to any fiber $\calH_y$, for $y$ in $B(x,r)$.

Choose $0 < r  < 1$ so that the stability constant (see Theorem III.H.1.7 of \cite{BridsonHaefliger}) for $(\MitraProjK^r,0)$--quasigeodesics in $\HH^2$ is less than $1$.   
For any $x,y$ in $\Graph{}$ with $d(x,y) \leq r$, it follows that the parallel transport line from $z$ in $\calA_y(\gamma)$ to $\mathfrak{f}_x(z)$ in $\calH_x$ is contained in $\calN_1(\calA(\gamma))$ and hence
	\[ 
		d_\gamma(z,\mathfrak{f}_x(z)) = d(z,\mathfrak{f}_x(z)) = d(x,y) \leq r .
	\]

Let $\delta_h$ denote the hyperbolicity constant for $\HH^2$.
\begin{claim}  
Given any two points $w,z$ in $\calH_{\Graph{}}$ with $d(w,z) \leq r$, we have
	\[ 
		d_\gamma(\closest{\gamma}(w),\closest{\gamma}(z)) \leq \MitraProjK^r r +  2(1+\MitraProjK^r \delta_h)  +  r.
	\]
\end{claim}
\begin{proof}[Proof of claim.]
Let $w,z$ in $\Graph{}$ be any two points with $d(w,z) \leq r$ and let $x = \pi(w)$ and $y = \pi(z)$ so that $d(x,y) \leq r$.

Recall that for any $c \geq \delta_h$ and any geodesic triangle $\triangle \subset \HH^2$, the set of points within a distance $c$ of all three sides is nonempty and has diameter at most $2c$.  
The closest point projection of one vertex of $\triangle$ to the opposite side is such a point.

Inside $\calH_y$, the point $\closest{\gamma}(z)$ is within $\delta_h$ of all three sides of the geodesic triangle $\triangle$ having vertex $z$ and opposite side $\calA_y(\gamma)$.    
It follows that inside $\calH_x$, the point $\mathfrak{f}_x\closest{\gamma}(z)$ has distance at most $\MitraProjK^r \delta_h$ from all three sides of the $(\MitraProjK^r,0)$--quasigeodesic triangle $\mathfrak{f}_x(\triangle)$.  Because the sides of this are within a distance $1$ of the geodesics with the same endpoints, it follows that $\mathfrak{f}_x\closest{\gamma}(z)$ is within a distance $1+ \MitraProjK^r\delta_h$ of all three sides of the geodesic triangle defined by $\mathfrak{f}_x(z)$ and $\calA_x(\gamma)$.  Since $\closest{\gamma}\mathfrak{f}_x(z)$ has distance at most $\delta_h < 1 + \MitraProjK^r\delta_h$ from each of these sides, it follows that
	\[ 
		d_x(\closest{\gamma}\mathfrak{f}_x(z),\mathfrak{f}_x\closest{\gamma}(z)) \leq 2(1 + \MitraProjK^r \delta_h). 
	\]
Moreover, the path exhibiting this distance bound lies entirely inside $\calH_x$, and the geodesic in $\calH_x$ between these points lies within a distance $1$ of $\calA_x(\gamma)$.  In particular, it follows that
	\[ 
		d_\gamma(\closest{\gamma}\mathfrak{f}_x(z),\mathfrak{f}_x\closest{\gamma}(z)) 
		\leq d_x(\closest{\gamma}\mathfrak{f}_x(z),\mathfrak{f}_x\closest{\gamma}(z)) 
		\leq 2(1 + \MitraProjK^r\delta_h). 
	\]

Applying the triangle inequality proves the claim, since
\begin{align}
	d_\gamma(\closest{\gamma}(w),\closest{\gamma}(z)) 
	& \leq  d_\gamma(\closest{\gamma}(w),\closest{\gamma}\mathfrak{f}_x(z)) + d_\gamma(\closest{\gamma} \mathfrak{f}_x(z),\mathfrak{f}_x\closest{\gamma}(z)) \notag \\
	& \quad \quad + d_\gamma(\mathfrak{f}_x\closest{\gamma}(z),\closest{\gamma}(z))\\
	& \leq d_x(w,\mathfrak{f}_x(z)) +  2(1+\MitraProjK^r\delta_h) + r\\
	& \leq \MitraProjK^r d(w,z) + 2(1+ \MitraProjK^r\delta_h) + r \label{EQ:TransportLipschitzInequality} \\
	& \leq \MitraProjK^r r +  2(1+\MitraProjK^r\delta_h)  +  r.
\end{align}
In inequality \eqref{EQ:TransportLipschitzInequality}, we have used the fact that $\mathfrak{f}_x$ is $\MitraProjK^r$--Lipschitz.
\end{proof}

From the claim we see that $\closest{\gamma}$ is $(\MitraQIK, \MitraQIC)$--coarsely Lipschitz, where $\MitraQIK = \MitraProjK^r  +  2(1+\MitraProjK^r\delta_h)/r  +  1$ and $\MitraQIC = \MitraProjK^r r +  2(1+\MitraProjK^r\delta_h)  +  r$.
Since the inclusion of $\calA(\gamma)$ into $\calH_{\Graph{}}$ is $1$--Lipschitz, it follows that $\calA(\gamma)$ is $(\MitraQIK, \MitraQIC)$--quasiisometrically embedded.
\end{proof}

\subsection{Quasiisometric sections} \label{S:qisection}

Let $E$ and $B$ be metric spaces and let $\pi\co E \to B$ be a $1$--Lipschitz map.
By a \textit{$(k,c)$--quasiisometric section} (or just \textit{$(k,c)$--section}) of $\pi\co E \to B$ we mean a subset $\Sigma \subset E$ that is the image of a $(k,c)$--coarsely Lipschitz map $\sigma\co B \to E$  with $\pi \circ \sigma = id_B$.   
Since $\pi$ is $1$--Lipschitz, the map $\sigma$ is a $(k,c)$--quasiisometric embedding.  
In fact,
	\[ 
		d(x,y)	 = d\big(\pi \sigma(x),\pi \sigma(y)\big) 
				\leq d\big(\sigma(x),\sigma(y)\big) 
				\leq k d(x,y) + c.
	\]

Mosher's Quasiisometric Section Lemma \cite{Mosher.1996} says that if $1 \to \pi_1(S) \to \Gamma \to G \to 1$ is hyperbolic, then there is a $(\MosherQISectionK,\MosherQISectionC)$--section of $\pi\co\Graph{\Gamma} \to \Graph{}$ for some $\MosherQISectionK$ and $\MosherQISectionC$.
From this we obtain a $(\MosherQISectionK,\MosherQISectionC)$--section $\Sigma$ of $\calH_{\Graph{}} \to \Graph{}$ after enlarging $\MosherQISectionK$ and $\MosherQISectionC$.
Using the fact that $\pi_1(S) < \Gamma$ acts cocompactly on the fibers, and by taking $\MosherQISectionC$ even larger, it follows that for any point $z$ in $\calH_{\Graph{}}$ there is a $(\MosherQISectionK,\MosherQISectionC)$--section $\Sigma$ for $\calH_{\Graph{}} \to \Graph{}$ containing $z$; see also \cite{Mj.Sardar.2012}. 

Given a $(\MosherQISectionK,\MosherQISectionC)$--section $\Sigma$ of $\calH_{\Graph{}} \to \Graph{}$, we have that $\closest{\gamma}(\Sigma)$ is a $(\OurQISectionK,\OurQISectionC)$--section for $\OurQISectionK  = \MosherQISectionK \MitraQIK$ and $\OurQISectionC  = \MitraQIK \MosherQISectionC + \MitraQIC$,  by Lemma \ref{L:qiladder}.  
We therefore have the following result of \cite{Mj.Sardar.2012}.

\begin{lemma}[Mj--Sardar \cite{Mj.Sardar.2012}]\label{L:qisectioninladder}
Given a hyperbolic sequence $1 \to \pi_1(S) \to \Gamma \to G \to 1$ with associated hyperbolic plane bundle $\HH^2 \to \calH_{\Graph{}} \to \Graph{}$, there are $\OurQISectionK $ and $\OurQISectionC $ with the following property.  For all $\gamma$ in $\pi_1(S)$, all $x$ in $\Graph{}$, and all $z$ in $\calA_x(\gamma)$ there exists a $(\OurQISectionK ,\OurQISectionC )$--section $\Sigma $ of $\calH_{\Graph{}} \to \Graph{}$ with $\Sigma \subset \calA(\gamma)$ and $\Sigma \cap \calH_x = \{z \}$. 
\qed
\end{lemma}

A section $\Sigma$ as in this lemma will be called a \textit{$(\OurQISectionK ,\OurQISectionC )$--section for $\gamma$ (though $z$)}.
In the sequel we are interested in collections of these.  
The leaf $\calA_x(\gamma)$ is a line oriented by the action of $\gamma$, and so possesses a well--defined order.
We say that a collection $\{\Sigma_n \}_{n \in \ZZ}$ of $(\OurQISectionK ,\OurQISectionC )$--sections for $\gamma$ are \textit{linearly ordered over $x$} if the assignment $n \mapsto \Sigma_n \cap \calA_x(\gamma)$ is order preserving.

\begin{theorem} \label{T:generalbundlewidth4}
Given a hyperbolic sequence $1 \to \pi_1(S) \to \Gamma \to G \to 1$ with associated hyperbolic plane bundle $\HH^2 \to \calH_{\Graph{}} \to \Graph{}$, there are $\GapOverShortest > \GapOverRelShortest > 0$ with the following property.  
If $\gamma$ in $\pi_1(S)$ is nonfilling and $\{ \Sigma_n \}_{n \in \ZZ}$ is a collection of $(\OurQISectionK ,\OurQISectionC )$--sections for $\gamma$ such that
	\[
		\{\Sigma_n\}_{n \in \ZZ} \mbox{ is linearly ordered over } x_\gamma \mbox{ and }d_{x_\gamma}(\Sigma_n,\Sigma_{n+1}) = \GapOverShortest, 
	\]
then, for \textbf{every} $x$ in $\Graph{}$,
	\[ 
		\{\Sigma_n\}_{n \in \ZZ} \mbox{ is linearly ordered over } x \mbox{ and } d_x(\Sigma_n,\Sigma_{n+1}) \geq \GapOverRelShortest. 
	\]
\end{theorem}

\subsubsection{Proof of Theorem \ref{T:generalbundlewidth3}\label{S:QIAxisFromSections} assuming Theorem \ref{T:generalbundlewidth4}.}
\begin{proof}[Proof of Theorem \ref{T:generalbundlewidth3} assuming Theorem \ref{T:generalbundlewidth4}.] Let $\gamma$ be nonfilling.  
By Lemma \ref{L:qisectioninladder}, there are $(\OurQISectionK ,\OurQISectionC )$--sections $\{\Sigma_n\}_{n \in \ZZ}$ for $\gamma$ as in Theorem \ref{T:generalbundlewidth4}.

Let $\calR_n$ denote the open region in $\calA(\gamma)$ between $\Sigma_n$ and $\Sigma_{n+1}$.
By the conclusion of Theorem \ref{T:generalbundlewidth4}, each $\calR_n$ is a union of intervals, one in each fiber.
According to Theorem 3.2 of \cite{Mj.Sardar.2012}, there are constants $K'$ and $C' $ depending only the bundle $\HH^2 \to \calH_{\Graph{}} \to \Graph{}$ such that the fiberwise closest point projection 
	\[
		\closest{n} \co \calH_{\Graph{}} \to \calR_n
	\]
is $(K',C')$--coarsely Lipschitz map (where $\calR_n$ is given the metric inherited from the path metric on a sufficiently large neighborhood in $\calH_{\Graph{}}$).
Theorem 3.2 of \cite{Mj.Sardar.2012} is attributed to Mitra \cite{Mitra.1998}, as it is a direct translation of arguments there, much like the proof of Lemma \ref{L:qiladder}.

Define
	\[ 
		\eta_\gamma\co\calA(\gamma) \to  \calA_{x_\gamma}(\gamma) 
	\]
by $\eta_\gamma(\calR_n) = \eta_\gamma(\Sigma_n) = \Sigma_n \cap \calA_{x_\gamma}(\gamma)$. 
We will show that $\eta_\gamma$ is coarsely Lipschitz.

\begin{claimplain}
There is a $\FewShortSectionsBound> 0$ depending only on the bundle $\HH^2 \to \calH_{\Graph{}} \to \Graph{}$ such that if $w$ is in $\calR_m \cup \Sigma_m$ and $z$ is in $\calR_n \cup \Sigma_n$ with $d(w,z) \leq 1$, then $|m-n| \leq \FewShortSectionsBound$.
\end{claimplain}
\begin{proof}[Proof of claim]
Assume that $m \leq n$.

First assume that $w$ and $z$ are in the same fiber $\calA_{\pi(w)}(\gamma) = \calA_{\pi(z)}(\gamma)$.
By Theorem \ref{T:generalbundlewidth4}, we have $d_{\pi(w)}(w,z) \geq \GapOverRelShortest  (n-m)$.
Now, the fibers of $\calH_{\Graph{}}$ (in which the fibers of $\calA(\gamma)$ are geodesic) are uniformly proper, and so there is a positive $\UniformPropernessConstantA$ depending only on $\HH^2 \to \calH_{\Graph{}} \to \Graph{}$ such that $d(w,z) \geq \UniformPropernessConstantA \, d_{\pi(w)}(w,z)$.
So 
	\[
		1 
		\geq d(w,z) 
		\geq \UniformPropernessConstantA \GapOverRelShortest  \, (n-m-1),
	\]
and we are done in this case with $B_1 = 1/ \UniformPropernessConstantA \GapOverRelShortest + 1$.

If $w$ and $z$ are in different fibers, we argue as follows.
Let $z'$ be a point in the fiber $\calH_{\pi(w)}$ with 
	\[
		d(z,z') 
		= d\left(z,\calH_{\pi(w)}\right) 
		\leq d(z,w) \leq 1.
	\]
We have $\closest{n}(z) = z$ and $\closest{n}(z') = z''$ for some $z''$ in $\calR_n \cap \calH_{\pi(w)}$.
Since $\closest{n}$ is $(K',C')$--coarsely Lipschitz, uniform properness gives us
	\begin{align*}
		1 + K' + C' 
			& \geq 1 + K' d(z,z') + C' 
			\\
			& \geq d(w,z) + d(z,z'')
			\\
			& \geq d(w,z'') 
			\\
			& \geq \UniformPropernessConstantA \GapOverRelShortest  \, (n-m-1),
	\end{align*}
and the proof is complete with $\FewShortSectionsBound = (1 + K'  + C')/\UniformPropernessConstantA \GapOverRelShortest + 1$.
\end{proof}

It follows from the claim that
	\[
		 d_{x_\gamma}(\eta_\gamma(z),\eta_\gamma(w)) \leq \FewShortSectionsBound  \GapOverShortest  
	\]
if $d(z,w) \leq 1$, and so $\eta_\gamma$ is $(\FewShortSectionsBound \GapOverShortest  ,\FewShortSectionsBound \GapOverShortest  )$--coarsely Lipschitz.
It follows that $\calA_{x_\gamma}(\gamma)$ is $(\FewShortSectionsBound \GapOverShortest  ,\FewShortSectionsBound \GapOverShortest  )$--quasiisometrically embedded in $\calA(\gamma)$, and hence $(K,C)$--quasiisomet-rically embedded in $\calH_{\Graph{}}$ for $K = \MitraQIK \FewShortSectionsBound \GapOverShortest  $ and $C = \MitraQIK \FewShortSectionsBound \GapOverShortest   + \MitraQIC$, by Lemma \ref{L:qiladder}.

This proves Theorem \ref{T:generalbundlewidth3}.
\end{proof}

For $x$ sufficiently far from $x_\gamma$, the distances $d_x(\Sigma_n,\Sigma_{n+1})$ are in fact much larger than the estimate in Theorem \ref{T:generalbundlewidth4}.  
As a function of $d(x,x_\gamma)$, they are exponentially larger than the distances $d_{x_\gamma}\big(\Sigma_n \cap \calA_{x_\gamma}(\gamma),\Sigma_{n+1} \cap \calA_{x_\gamma}(\gamma)\big)$, due to \textit{flaring}.  
For nonfilling $\gamma$, the exponential growth will kick in outside a ball about $x_\gamma$ of a uniformly bounded radius.  

The rest of the paper is devoted to the proof of Theorem \ref{T:generalbundlewidth4}, which requires a study of quadratic differentials, Teichm\"uller geodesics, and singular \textsc{Sol}  metrics, taken up in the next section.

\subsection{Quadratic differentials and flat metrics} \label{S:qd and flat}

We refer the reader to \cite{Strebel} for a detailed treatment of quadratic differentials and their associated flat metrics.

Given a complex structure on $S$, a unit--norm holomorphic quadratic differential $q$ on $S$ both determines and is determined by a nonpositively curved Euclidean cone metric on $S$ together with a pair of orthogonal singular foliations with geodesic leaves (called the \textit{vertical} and \textit{horizontal} foliations).  
Given $q$ and a nonsingular point $p$, there is a \textit{preferred coordinate} $\zeta = x+iy$ which carries a neighborhood of $p$ isometrically into the plane such that the arcs of the horizontal and vertical foliations to horizontal and vertical segments, respectively.

We let $\calQ^1(S)$ denote the space of all unit--norm holomorphic quadratic differentials on $S$, which forms the unit cotangent bundle over Teichm\"uller space $\Teich(S)$.  
We let $m = m(q)$  denote the hyperbolic metric in the conformal class of a quadratic differential $q$, and write $q \mapsto m(q)$ for the map $\calQ^1(S) \to \Teich(S)$.

Let $\widetilde S \to S$ be the universal covering.  
Given $q$ in $\calQ^1(S)$, we abuse notation and continue to refer to  the pullback of $q$ and $m$ to $\widetilde S$ as $q$ and $m$, respectively.  
The identity map $id_{\widetilde S}\co \widetilde S \to \widetilde S$ is a quasiisometry with respect to $m$ and the singular flat metric for $q$.  In fact, by Proposition 2.5 of \cite{Farb.Mosher.2002} or Lemma 3.3 of \cite{Minsky.1994}, for example, we have the following lemma.
\begin{lemma}[Minsky \cite{Minsky.1994}] \label{L:FMuniformqi}
Given $r > 0$ there exist $\HypFlatQIK,\HypFlatQIC> 0$ such that if $q$ in $\calQ^1(S)$ lies over the $r$--thick part of $\Teich(S)$, then
	\[ 
		id_{\widetilde S}\co(\widetilde S, m) \to (\widetilde S, q) 
	\]
is a $(\HypFlatQIK,\HypFlatQIC)$--quasiisometry.
\qed
\end{lemma}

\subsubsection{Geodesics and straight segments.}

Fix $q$ in $\calQ^1(S)$.   Given $\gamma$ in $\pi_1(S)$ a (nontrivial) element we will let $\gamma_0^*$ denote the $q$--geodesic representative in $S$ and $\widetilde \gamma_0^*$ a lift of this geodesic to a biinfinite $q$--geodesic in $\widetilde S$.   The geodesic $\gamma_0^*$ should be considered a locally isometric map from a circle or interval of some length into $S$ as the geodesic is not determined by its image.

The geodesics $\gamma_0^*$ and $\widetilde \gamma_0^*$ are either Euclidean geodesics (geodesics in the complement of the singularities) or concatenations of \textit{straight segments} (Euclidean geodesic segments connecting pairs of singular points with no singular points in their interior).

We let $\| \gamma \|_q$ denote the $q$--length of $\gamma_0^*$ and $\| \gamma \|_{q,v}$ and $\| \gamma \|_{q,h}$ the vertical and horizontal lengths of $\gamma_0^*$, respectively.  These are related by
	\begin{align}\label{E:length,vertical,horizontal} 
		\frac{1}{2} ( \|\gamma\|_{q,v} + \|\gamma \|_{q,h} ) 
			& \leq \max \{ \|\gamma\|_{q,v} , \|\gamma \|_{q,h} \}  \\
			& \leq \| \gamma \|_q \\
			& \leq \|\gamma\|_{q,v} + \|\gamma \|_{q,h} \\
			& \leq 2 \max \{ \|\gamma\|_{q,v} , \|\gamma \|_{q,h} \}.
	\end{align}

More generally, given a (local) $q$--geodesic $\delta \co I \to S$ or $\delta \co I \to \widetilde S$ defined on an interval $I \subset \RR$, we let $\| \delta \|_q$, $\| \delta \|_{q,h}$, and $\| \delta \|_{q,v}$ denote the length, horizontal length, and vertical length, respectively.

We let $\| \gamma \|_m$ denote the length of the $m=m(q)$--geodesic representative.
Given $r > 0$, if $\HypFlatQIK ,\HypFlatQIC$ are as in Lemma \ref{L:FMuniformqi}, we have 
	\begin{equation} \label{E:Kcomparelengths} 
			\frac{1}{\HypFlatQIK } \| \gamma \|_q 
			\leq \| \gamma \|_m 
			\leq \HypFlatQIK  \| \gamma \|_q.
	\end{equation}
The inequality \eqref{E:Kcomparelengths} is free of the constant $\HypFlatQIC$ thanks to the fact that the length is equal to the asymptotic translation length.

More generally, given any geodesic metric $m'$ on $S$ for which the pullback to $\widetilde S$ makes $id_{\widetilde S} \co (\widetilde S,m') \to (\widetilde S,q)$ a $(\HypFlatQIKgen,\HypFlatQICgen)$--quasiisometry, then
	\begin{equation} \label{E:Kcomparelengths2}
			\frac{1}{\HypFlatQIKgen } \| \gamma \|_q 
			\leq \| \gamma \|_m 
			\leq \HypFlatQIKgen  \| \gamma \|_q.
	\end{equation}

From (\ref{E:Kcomparelengths}) we easily obtain the following.
\begin{lemma} \label{L:longstraightsegment}
For any $r > 0$, there exists $\epsilon > 0$ with the following property.  Given any $q$ in $\calQ^1(S)$ lying over the $r$--thick part of $\Teich(S)$ and any (local) $q$--geodesic segment $\delta\co [0,1] \to S$ or $\delta\co [0,1] \to \widetilde S$, there is an arc of $\delta$ of length at least $\epsilon$ containing no singularities.

\end{lemma}
\begin{proof}
We assume as we may that $r < 1$ and set $\epsilon = r/(\HypFlatQIK (4g-2)) < 1/(4g-2)$.

Suppose that there is a $q$--geodesic segment $\delta \co  [0,1] \to S$ such that every subsegment of length at least $\epsilon$ contains a singularity.
This segment contains a concatenation $\delta'$ of at least $4g-4$ straight segments of $q$--length less than $\epsilon$, each connecting a pair of singularities.  
Since there are at most $4g-4$ singularities of $q$, the segment $\delta'$ must visit some singularity more than once, thus forming a loop $\beta$ of $q$--length less than $(4g-4)\epsilon < r/\HypFlatQIK $.   
Except at the basepoint, this loop $\beta$ is locally geodesic, and is therefore essential.  
By \eqref{E:Kcomparelengths}, the hyperbolic length of $\beta$ is less than $\HypFlatQIK (r/\HypFlatQIK ) = r$, which contradicts the fact that $q$ lies over the $r$--thick part of $\Teich(S)$.  

For $\delta\co [0,1] \to \widetilde S$, we push forward to $S$ and appeal to the first case.
\end{proof}

Applying the lemma to any closed geodesic $\gamma_0$ we have the following.

\begin{corollary} \label{C:longstraightsegment}
Let $r>0$ and let $\epsilon$ be as in Lemma \ref{L:longstraightsegment}.
If $q$ in $\calQ^1(S)$ lies over the $r$--thick part of $\Teich(S)$ and $\gamma$ in $\pi_1(S)$, then $\gamma_0$ contains a straight segment of length at least $\epsilon$. 
\qed
\end{corollary}

\subsection{Teichm\"uller geodesics and lengths} \label{S:teichmuller geodesics and lengths}

We refer the reader to \cite{Abikoff.1980} and \cite{Hubbard.2006} for detailed treatments of Teichm\"uller theory.

\subsubsection{Teichm\"uller deformations.} \label{S:teichmuller geodesics}

The Teichm\"uller deformation associated to a quadratic differential $q$ in $\calQ^1(S)$ determines a $1$--parameter family of quadratic differentials $q_t$.  
More precisely, if $q$ has preferred coordinate $\zeta = x+ iy$, then $q_t$ is determined by its preferred coordinate $\zeta_t = e^t x + i e^{-t} y$ (in particular, $q = q_0$).  
The map $\tau_q\co \RR \to \Teich(S)$ obtained by composing $t \mapsto q_t$ with the projection $\calQ^1(S) \to \Teich(S)$, namely $\tau_q(t) = m_t = m(q_t)$, is a \textit{Teichm\"uller geodesic}.    
Every geodesic in $\Teich(S)$ is of this form.

\subsubsection{Balance times}\label{S:Balance.Times}

If $\delta\co I \to S$ or $\delta\co I \to \widetilde S$ is a (local) $q$--geodesic, we can reparameterize $\delta$ to be a (local) $q_t$--geodesic for any $t$.  
In particular, straight segments can be linearly reparameterized to be (locally) geodesic.  
We denote the reparameterization by $\delta_t$.

For any $\gamma$ in $\pi_1(S)$ we have
	  \[
	  	\|\gamma \|_{q_t,h}   
			= \| \gamma \|_{q,h}  \, e^t \mbox{ and }    \| \gamma \|_{q_t,v} 
			=  \| \gamma \|_{q,v} \,  e^{-t}.
	\]
We let $\gamma_t^*$ and $\widetilde \gamma_t^*$ denote the $q_t$--geodesic reparameterizations of the $q_t$--geodesics $\gamma_0^*$ and $\widetilde \gamma_0^*$, respectively.

We say that $\gamma$ is \textit{balanced at time $t$} if $\| \gamma \|_{q_t,h} = \| \gamma \|_{q_t,v}$.  
If $\gamma$ is balanced at time $t_0$, then for $b = \|\gamma\|_{q_{t_0},v} + \|\gamma \|_{q_{t_0},h}$, we have
\begin{equation} \label{E:euclengthcosh} b \cosh(t-t_0) \leq \| \gamma \|_{q_t} \leq 2b \cosh(t-t_0)
\end{equation}
by \eqref{E:length,vertical,horizontal}.
So $\| \gamma \|_{q_t}$ is minimized  in the interval 
$\big[t_0-\arccosh(2), \, t_0+\arccosh^{-1}(2)\big]$ and grows exponentially in $|t|$.

Given any $q$, suppose $m'_t$ is a $1$--parameter family of hyperbolic metrics on $S$ for which $id_{\widetilde S} \co (\widetilde S,m'_t) \to (\widetilde S,q_t)$ is a $(\HypFlatQIKgen,\HypFlatQICgen)$--quasiisometry.  Then 
\begin{equation} \label{E:hyplengthcosh} 
\frac{b}{\HypFlatQIKgen} \cosh(t - t_0) \leq \| \gamma \|_{m_t} \leq 2b\HypFlatQIKgen  \cosh(t - t_0)
\end{equation}
by \eqref{E:Kcomparelengths2} and \eqref{E:euclengthcosh}.
In particular, the $m'_t$--length along $\tau_q(t)$ is minimized in the interval 
$\big[t_0 - \arccosh(2\HypFlatQIKgen^2), \, t_0+ \arccosh(2\HypFlatQIKgen^2)\big]$.

As an example, we could take $m'_t = m_t = m(q_t)$ to be the underlying hyperbolic metric, and then $(\HypFlatQIKgen,\HypFlatQICgen) = (\HypFlatQIK,\HypFlatQIC)$ by Lemma \ref{L:FMuniformqi}.
However, Theorem \ref{T:hyperbolictosolv} below provides our primary example of interest.

\subsubsection{Vertical and horizontal.}

Given $\epsilon > 0$, $0 < \theta < \pi/4$ and $q$ in $\calQ^1(S)$, we say that a $q$--straight segment $\delta$ is \textit{$\theta$--almost vertical} (respectively, \textit{$\theta$--almost horizontal}) \textit{with respect to $q$} if it makes an angle less than $\theta$ with the vertical (respectively, horizontal) direction.  
A closed geodesic $\gamma_0^*$, or its lift $\widetilde \gamma_0^*$, is called \textit{$(\epsilon,\theta)$--almost vertical} (respectively, \textit{$(\epsilon,\theta)$--almost horizontal}) \textit{with respect to $q$} provided it is a concatenation of $q$--straight segments each of which is $\theta$--almost vertical (respectively, $\theta$--almost horizontal), or has length less than $\epsilon$.  
Subject to certain constraints described below, the constants $\epsilon$ and $\theta$ will be fixed, and we will thus refer to segments and geodesics as simply almost vertical or almost horizontal.  
The discussion here differs from that of \cite{MasurMinsky.1999} in that the constraints we consider  depend on the thickness constant $r > 0$.

\subsubsection{Nonfilling curves after Masur and Minsky.}

The next proposition relies heavily on the work of Masur and Minsky, specifically Sections 5 and 6 of \cite{MasurMinsky.1999}.  In particular, Masur and Minsky place an upper bound on $\epsilon$ and $\theta$, depending only on $\chi$, that dictates, among other things, the amount of time it takes for a balanced geodesic to become almost horizontal.  We henceforth assume that $\ourepsilon,\ourtheta$ are less than this bound.  For any fixed $r > 0$, we also assume that $\epsilon_0$ is less than the constant $\epsilon$ coming from Lemma \ref{L:longstraightsegment}.


\begin{proposition} \label{P:nonfillinggrowth} Given $r > 0$, there is a  $\OurT > 0$ with the following property.  Suppose $q$ in $\calQ^1(S)$ defines an $r$--thick geodesic $\tau_q$ in $\Teich(S)$ and $\gamma$ in $\pi_1(S)$ is nonfilling, balanced at time $0$ in $\RR$.   For any geodesic subpath $\delta_0 \subset \widetilde \gamma_0^*$ with $\| \delta_0 \|_q > e^\OurT$ we have
	\[ 
		\| \delta_t \|_{q_t} > \frac {\ourepsilon e^{|t| - \OurT}}{4} \| \delta_0  \|_q = \frac{\ourepsilon e^{-\OurT}}{4}  e^{|t|}  \| \delta_0 \|_q
	\]
for any $t$.
\end{proposition}
We note the similarity between the conclusion of this proposition and \eqref{E:euclengthcosh}.  By comparison, \eqref{E:euclengthcosh} is a statement about the $q_t$--length of the entire curve $\gamma$, while this proposition provides information about the $q_t$--length of any definite length segment of $\gamma_0^*$.  In particular, it also grows exponentially outside some neighborhood of the balance time.  Furthermore, while \eqref{E:euclengthcosh} is true for any closed geodesic, Proposition \ref{P:nonfillinggrowth} is false if one allows $\gamma$ to be filling: there is no $T$ making the proposition valid for all filling $\gamma$.

\begin{proof}[Proof of Proposition \ref{P:nonfillinggrowth}]
In what follows, we appeal to Lemmas 6.4 and 6.5 of \cite{MasurMinsky.1999}, which provide bounds on diameters of shadows in the curve complex $\calC(S)$ of certain subsets of the Teichm\"uller geodesic $\tau_q$.  
Since ours is an $r$--thick geodesic, the shadow is a uniform quasigeodesic.
This is Lemma 4.4 of \cite{Rafi.Schleimer.2009}. 
It also follows quickly from the main theorem of \cite{Minsky.1996} (see Section 7.4 of \cite{Kent.Leininger.2008}).  
We may therefore turn bounds on diameters in $\calC(S)$ into bounds on diameters in the domain $\RR$ of $\tau_q$, and we do so without further comment.

Since $\gamma$ is nonfilling, there is an essential simple closed curve $\alpha$ disjoint from it.
Let $t_0$ denote the balance time for $\alpha$.
\begin{claim} There exists $\NaughtT > 0$, depending only on $\ourepsilon$, $\ourtheta$, and $r$ such that $\gamma_t^*$ is almost horizontal for all $t > \NaughtT$ and is almost vertical for all $t < -\NaughtT$.
\end{claim}
\begin{proof}
By Lemma 6.5 of \cite{MasurMinsky.1999}, there is a $\FirstT > t_0$ such that $\FirstT - t_0$ is bounded by a constant $B(\ourepsilon,\ourtheta,r)$ and such that for all $t > \FirstT$, the geodesic $\alpha_t^*$ is almost horizontal.   
Since $i(\delta,\alpha) = 0$, no segment of $\gamma_t^*$ intersects any segment of $\alpha_t^*$ away from the singularities.  
Pick a straight segment of $\alpha_t^*$ with length at least $\ourepsilon$ (from Corollary \ref{C:longstraightsegment}).
As in the last paragraph of the proof of Lemma 6.5 of \cite{MasurMinsky.1999}, we can appeal to Lemma 6.4 of \cite{MasurMinsky.1999} to find a $\SecondT > \FirstT$ such that, for all $t > \SecondT$, the geodesic $\gamma_t^*$ is almost horizontal.\footnote{The key to the proof of Lemma 6.5 of \cite{MasurMinsky.1999} is finding a disjoint almost horizontal straight segment.
In our setting, this is provided by a segment of $\alpha_t^*$.}
Moreover, the distance $\SecondT - \FirstT$, and hence also $\SecondT - t_0$, is bounded by a constant  $C(\ourepsilon,\ourtheta,r)$.

Reversing the roles of horizontal and vertical, there is $\ThirdT < t_0$ such that $\gamma_t^*$ is almost vertical for all $t < \ThirdT$, and $t_0 - \ThirdT$ is bounded by some $D(\ourepsilon,\ourtheta,r)$.  The balance time $0$ for $\gamma$ must occur in the interval $[\ThirdT,\SecondT]$ (since $\gamma$ is neither almost vertical nor almost horizontal when it is balanced), and setting $\NaughtT = \max \{ \SecondT, |\ThirdT| \}$ proves the claim.
\end{proof}

For all $t > 0$, we have
	\begin{equation} \label{E:exponentialgrowthqd} 
		\| \delta_t \|_{q_t} \geq e^{-t} \| \delta_0 \|_{q_0} .
	\end{equation}
For $t = \NaughtT$, we have
	\[ 
		\| \delta_{\NaughtT} \|_{q_{\NaughtT}} \geq e^{-\NaughtT} \| \delta_0 \|_{q_0},
	\]
and we set $\OurT = 2\NaughtT$.

Now, if $\delta_0 \subset \widetilde \gamma_{\NaughtT}^*$ is a straight segment of length at least $e^\OurT$ we have
	\[ 
		\| \delta_{\NaughtT} \|_{q_{\NaughtT}}  \geq e^{-\NaughtT} \| \delta_0 \|_{q_0} \geq e^{-\NaughtT} e^\OurT > 1. 
	\]
Therefore, by Lemma \ref{L:longstraightsegment}, the segment $\delta_{\NaughtT}$ contains a segment $\delta_{\NaughtT}'$ of length at least $\ourepsilon$ contained in a straight segment.  
This segment $\delta_{\NaughtT}'$ must be almost horizontal since $\gamma_{\NaughtT}^*$ (and hence $\widetilde \gamma_{\NaughtT}^*$) is almost horizontal.  Therefore, for all $t \geq \NaughtT$ we have
		\[
			 \| \delta_t' \|_{q_t} 		
			\geq \| \delta_t' \|_{q_t,h}
			\geq e^{t-\NaughtT} \| \delta_{\NaughtT}' \|_{q_{\NaughtT},h} 	
			\geq \frac{e^{t-\NaughtT}}{2} \| \delta_{\NaughtT}' \|_{q_{\NaughtT}} 			\geq \frac{\ourepsilon e^{t-\NaughtT}}{2}
		\]
There is such a segment $\delta_{\NaughtT}'$ in each segment of length $1$ in $\delta_{\NaughtT}$.  By subdividing $\delta_{\NaughtT}$ into a maximal number $n$ of disjoint segments of length at least $1$, so that $n \leq \| \delta_{\NaughtT} \|_{q_\NaughtT} < n+1$, we have
		\[
			\| \delta_t \|_{q_t} \geq \frac{n \ourepsilon e^{t-\NaughtT}}{2}
			= \frac{n}{n+1} \frac{(n+1) \ourepsilon e^{t-\NaughtT}}{2}
			\geq \frac{\ourepsilon e^{t-\NaughtT}}{4} \| \delta_{\NaughtT} \|_{q_\NaughtT}
		\]
Combining these strings of  inequalities we see that, for $t \geq \NaughtT$, we have
	\[ 
		\| \delta_t \|_{q_t} 
		\geq \frac{\ourepsilon e^{t-\NaughtT}}{4} e^{-\NaughtT} \| \delta_0 \|_{q_0} 
		= \frac{\ourepsilon e^{t-\OurT}}{4} \| \delta_0 \|_{q_0}.
	\]
On the other hand, if $0 \leq t < \NaughtT$, then $-t > t - \OurT$.
Since $\ourepsilon/4 < 1$, we therefore have 
	\[ 
		\| \delta_t \|_{q_t} 
		\geq  e^{-t} \| \delta_0 \|_{q_0} 
		\geq  e^{t-\OurT} \| \delta_0 \|_{q_0}  
		\geq   \frac{\ourepsilon e^{t-\OurT}}{4} \| \delta_0 \|_{q_0} .
	\]
by (\ref{E:exponentialgrowthqd}).
Thus the proposition follows for $t \geq 0$.  A symmetric argument proves the proposition for $t \leq 0$.
\end{proof}

\subsection{Surface bundles over Teichm\"uller geodesics}

\subsubsection{Singular \textsc{Sol} and hyperbolic metrics are uniformly quasiisometric} \label{S:singular sol}

Given $q$ in $\calQ^1(T)$ with Teichm\"uller geodesic $\tau_q$, consider the pullback bundle
	\[ 
		\xymatrix{  \HH^2 \ar[r] & \calH_{\tau_q} \ar[r] & \tau_q .} 
	\]
The lifted quadratic differential $q_t$ defines a flat metric on the fiber $\calH_{\tau_q(t)} \cong \HH^2$.
The lifted Teichm\"uller mapping identifies the fibers $\calH_{\tau_q(t)}$ with $\calH_{\tau_q(0)}$, determining a homeomorphism $\calH_{\tau_q(t)} \cong \widetilde S \times \RR$ so that $(z,0) \mapsto (z,t)$ is the Teichm\"uller mapping.   The coordinate $t$ and preferred coordinates $\zeta = x+iy$ for $q$ give local coordinates for $S \times \RR$ away from $\{ \mbox{singularities of $q$} \} \times \RR$.  
We thus have a metric $e^{2t} dx^2 + e^{-2t} dy^2 + dt^2$ on $(S - \{ \mbox{singularities of $q$} \})\times \RR$ whose metric completion is naturally identified with $\widetilde S \times \RR \cong \calH_{\tau_q}$, and whose restriction to each fiber is just the metric $q_t$.  
We let $\calH^\Sol_{\tau_q}$  denote $\calH_{\tau_q}$ with this metric.
This is the \textit{singular \textsc{Sol} metric associated to $q$}.

We now note that Proposition \ref{P:nonfillinggrowth} provides an ``exponential growth'' version of Theorem \ref{T:generalbundlewidth4} for the singular  \textsc{Sol}  metric.  
Given $\gamma$ in $\pi_1(S)$, define \textit{isometric} sections $\{ \isosection_n \}_{n \in \ZZ}$ of $\calH^\Sol_{\tau_q} \to \tau_q$ by picking linearly ordered points $\{ z_n \}_{n \in \ZZ} = \{(z_n,0) \}_{n \in \ZZ}  \subset \widetilde \gamma_0^* \subset  \widetilde S \times \{ 0 \}$.  
Let $\isosection_n = \{ (z_n,t) \mid t \in \RR \} \subset \calH^\Sol_{\tau_q} \cong \widetilde S \times \RR$. 
By construction, the $\isosection_n$ are linearly ordered over every $\tau_q(t)$.   
Let $\delta^n_0$ denote the segment from $z_n$ to $z_{n+1}$ inside $\widetilde \gamma_0^*$, so that $\delta^n_t$ is the segment from $\isosection_n$ to $\isosection_{n+1}$ inside $\widetilde \gamma_t^*$.
This gives us the following singular  \textsc{Sol}  variant of Theorem \ref{T:generalbundlewidth4}.

\begin{proposition}\label{P:isometricsectionsflaring}
Given $r > 0$, let $\OurT > 0$ be as in Proposition \ref{P:nonfillinggrowth}.  
Let $q$ be a unit--norm quadratic differential defining an $r$--thick geodesic $\tau_q$ in $\Teich(S)$ and suppose that $\gamma$ in $\pi_1(S)$ is nonfilling and balanced at time zero.  
Given isometric sections $\{\isosection_n \}_{n \in \ZZ}$ as above with
	\[ 
		d_{\tau_q(0)}(\isosection_n,\isosection_{n+1}) = \| \delta^n_0 \|_{q_0} \geq e^\OurT ,	\]
we have
	\[ 
		d_{\tau_q(t)} (\isosection_n,\isosection_{n+1}) \geq \frac{\ourepsilon e^{-\OurT}}{4} e^{|t|} d_{\tau_q(0)}(\isosection_n,\isosection_{n+1}). 
	\quad \quad \quad \quad \quad \quad \quad \quad \quad \quad \quad \quad
	\qed
	\]
\end{proposition}

Given a unit--norm quadratic differential $q$ defining an $r$--thick geodesic $\tau_q$ in $\Teich(S)$ and a nonfilling $\gamma$ in $\pi_1(S)$, the space $\calA^\Sol(\gamma) = \cup \, \widetilde \gamma_t^*$ is $\delta^\Sol$--hyperbolic for some  $\delta^\Sol = \delta^\Sol(g,r)$.
In fact, this space is quasiisometric to the hyperbolic plane.
Following the argument (in Section \ref{S:qisection}) that derives Theorem \ref{T:generalbundlewidth3} from Theorem \ref{T:generalbundlewidth4}, we have the following corollary of Proposition \ref{P:isometricsectionsflaring}.

If $[a,b]$ is an interval, we let
	\[
		\calA^\Sol_{[a,b]} = \bigcup_{a \leq t \leq b } \widetilde \gamma_t^* .
	\]

\begin{corollary}\label{C:quasiconvexfiberinSOL}
Let $r > 0$ and let $\OurT$, $q$, and $\gamma$ be as in Proposition \ref{P:isometricsectionsflaring}.  
There are constants $\SOLfiberQCXconstantA$, $\SOLfiberQIconstantK$, and $\SOLfiberQIconstantC$ depending only on $r$ and the genus $g$ of $S$ such that the fiber $\widetilde \gamma_0$ is a $(\SOLfiberQIconstantK, \SOLfiberQIconstantC)$--quasigeodesic in $\calA^\Sol(\gamma)$ and $\calA^\Sol_{[-a,a]}$ is $\SOLfiberQCXconstantA$--quasiconvex for all $a$.
\qed
\end{corollary}

Proposition \ref{P:isometricsectionsflaring} also has the following corollary.
\begin{corollary}\label{C:neighborhoodofsectionsinSOL}
Let $R, r > 0$ and let $\OurT$, $q$, $\gamma$, and $\isosection_n$ be as in Proposition \ref{P:isometricsectionsflaring}.  
There is an $\BoundOnSectionsInNeighborhood=\BoundOnSectionsInNeighborhood(R,r)$ such that if the $R$--neighborhood of $\isosection_n$ intersects $\isosection_m$, then $|n-m| \leq \BoundOnSectionsInNeighborhood$.
\qed
\end{corollary}

We now promote Proposition \ref{P:isometricsectionsflaring} to a statement about arbitrary $(k,c)$--sections.

\begin{proposition} \label{P:qisectionsingsolv}
Given $r,k,c > 0$, there exists $\GapOverRelBalance > \GapOverBalance > 0$ with the following property.   
Let $q$ be a unit--norm quadratic differential defining an $r$--thick geodesic $\tau_q$ in $\Teich(S)$ and suppose that $\gamma$ in $\pi_1(S)$ is nonfilling and balanced at time zero.  
Suppose that $\{ \Sigma_n \}_{n \in \ZZ}$ are $(k,c)$--sections contained in $\calA^\Sol(\gamma) = \cup_t \widetilde \gamma_t^*$ such that
	\[ 
		\{\Sigma_n\}_{n \in \ZZ} \mbox{ is linearly ordered over } \tau_q(0) \mbox{ and }d_{\tau_q(0)}(\Sigma_n,\Sigma_{n+1}) \geq \GapOverBalance. 
	\]
Then
	\[
		 \{\Sigma_n\}_{n \in \ZZ} \mbox{ is linearly ordered over } \tau_q(t) \mbox{ and } d_{\tau_q(t)}(\Sigma_n,\Sigma_{n+1}) \geq \GapOverRelBalance e^{|t|} 
	 \]
for every $t$ in $\RR$.
\end{proposition}
\begin{proof}
Let $\isosection_n$ be the isometric sections as in Proposition \ref{P:isometricsectionsflaring}.
By Proposition \ref{P:isometricsectionsflaring}, it suffices to show that there is a number $B$ such that if $\Sigma$ is a $(k,c)$--section contained in $\calA^\Sol(\gamma)$, then there are numbers $n > m$ with $n-m \leq B$ such that $\Sigma$ lies in the region bounded bounded by $\isosection_m$ and $\isosection_n$.

Let $\Sigma$ be a $(k,c)$--section contained in $\calA^\Sol(\gamma)$.
Let  $n > m$ be such that $\isosection_n$ and $\isosection_m$ intersect $\Sigma$ nontrivially.

Pick $(z_m, t_m)$ in $\isosection_m \cap \Sigma$ and $(z_n,t_n)$ in $\isosection_n \cap \Sigma$.
Let $(w_n, t_m)$ be the point in $\isosection_n \cap \widetilde \gamma_{t_m}^*$.

Assume that $0 \leq t_m \leq t_n$.

Let $\calG_\Sigma \co [0,j] \to \calA^\Sol$ be a $(k,c)$--quasigeodesic in $\Sigma$ joining $(z_m, t_m)$ and $(z_n,t_n)$.
Let $\calG_\isosection$ be the geodesic in $\isosection_n$ joining $(w_n,t_m)$ and $(z_n, t_n)$, let $\calV$ be a geodesic in $\calA^\Sol(\gamma)$ joining $(z_m, t_m)$ and $(w_n, t_m)$.

By Corollary \ref{C:quasiconvexfiberinSOL}, the set $\calA^\Sol_{[-t_m, t_m]}$ is $\SOLfiberQCXconstantA$--quasiconvex.
So $\calV$ lies in $\SOLfiberQCXconstantA$--neighborhood of $\calA^\Sol_{[-t_m, t_m]}$.

As the space $\calA^\Sol(\gamma)$ is $\delta^\Sol$--hyperbolic, it follows that the quasigeodesic triangle $\triangle = \calG_\Sigma \cup \calG_\isosection \cup \calV$ is $\delta'$--thin for some $\delta'$ depending only on $\delta^\Sol$ and $k$ and $c$.

 Let $\delta'' = 3 \max\{\SOLfiberQCXconstantA, \delta'\}$.
Since $\Sigma$ is a $(k,c)$--section, there is an $i = i(k,c)$ such that 
	\[
		\calG_\Sigma \big|_{[i, j]} \subset \calA^\Sol_{[t_m 
		+ \delta'', \,\scriptinfinity]}.
	\]
Since $\triangle$ is $\delta'$--thin and $\calV$ is contained in $\calA^\Sol_{[-\scriptinfinity, \, t_m + \SOLfiberQCXconstantA]}$, the segment $\calG_\Sigma \big|_{[i, j]}$ must lie in the $\delta'$--neighborhood of $\calG_\isosection$.
So $\calG_\Sigma$ lies in the $(ki + c + \delta')$--neighborhood of $\calG_\isosection \subset \isosection_n$.

Corollary \ref{C:neighborhoodofsectionsinSOL} now bounds $n-m$.

The cases $0 \leq t_n \leq t_m$, $t_m \leq t_n \leq 0$ and $t_n \leq t_m \leq 0$ are proven by essentially the same argument.
The cases $t_n \leq 0 \leq t_m$ and $t_m \leq 0 \leq t_n$ are proven by breaking $\calG_\Sigma$ into ``positive" and ``negative" segments, and running the above argument on each half.
\end{proof}

The following theorem is due to Farb and Mosher (see Proposition 4.2 of \cite{Farb.Mosher.2002} and its proof there), and is the final piece needed to prove Theorem \ref{T:generalbundlewidth4}.

\begin{theorem}[Farb--Mosher \cite{Farb.Mosher.2002}] \label{T:hyperbolictosolv}
Given $r, k,c > 0$, there exist $\FiberwiseQIK, \FiberwiseQIC$ with the following properties.  
Suppose $g\co \mathbb R \to \Teich(S)$ is a $(k,c)$--quasigeodesic that stays a uniformly bounded distance from the $r$--thick Teichm\"uller geodesic $\tau_q$ and let $\nu\co \RR \to \RR$ be a map so that $g(t) \mapsto \tau_q(\nu(t))$ is the closest point projection.  
Then this closest point projection is $(\FiberwiseQIK,\FiberwiseQIC)$--coarsely Lipschitz and lifts to a fiber--preserving $(\FiberwiseQIK,\FiberwiseQIC)$--quasiisometry
	\[ 
		\calH_g \to \calH^\Sol_{\tau_q} 
	\]
for which the maps on fibers $\HH_{g(t)} \to (\widetilde S,q_{\nu(t)})$ are $(\FiberwiseQIK, \FiberwiseQIC)$--quasiisometries.
\qed
\end{theorem}

\begin{proof}[Proof of Theorem \ref{T:generalbundlewidth4}]

To simplify the discussion, we suppress many of the constants implicit in the proof, and use ``uniform" and ``uniformly" to mean that the constants involved depend only on the sequence $1 \to \pi_1(S) \to \Gamma \to G \to 1$ and its associated bundle $\HH^2 \to \calH_{\Graph{}} \to \Graph{}$.

Let $\Sigma_n $ be our $(\OurQISectionK,\OurQISectionC)$--sections of $\calH_{\Graph{}} \to \Graph{}$.

For every $x$ in $\Graph{}$, take a biinfinite geodesic $\calG_0$ in $\Graph{}$ through $x$ and $x_\gamma$.
Composing with $\Graph{} \to \Teich(S)$ we get a uniformly quasigeodesic $\calG$ fellow travelling an $r$--thick Teichm\"uller geodesic $\tau_q$ for some $r = r(\Gamma)$.
We apply Theorem \ref{T:hyperbolictosolv} to produce a uniform fiber--preserving quasiisometry $\calH_\calG \to \calH_{\tau_q}^\Sol$.  
Pushing the $\Sigma_n\big|_\calG$ over to $\calH_{\tau_q}^\Sol$ we obtain uniformly quasiisometric sections $\Sigma_n'$.  
We apply Proposition \ref{P:qisectionsingsolv}, and push the conclusion back to $\calH_\calG$.  
The result is a statement identical to that of Theorem \ref{T:generalbundlewidth4} except that $x_\gamma$ has been replaced with the pullback $x_0$ of the balance time $\tau_q(0)$.
Setting $m_t' = g(t)$ and $\tau_q(\nu(t))$ (with the appropriate reparameterization) in the discussion at the end of Section \ref{S:Balance.Times}, we have $(\HypFlatQIKgen,\HypFlatQICgen) = (\FiberwiseQIK,\FiberwiseQIC)$, so that (\ref{E:hyplengthcosh}) implies that $x_0$ is uniformly close to $x_\gamma$, and this completes the proof.
\end{proof}

{\small
\bibliographystyle{plain}
\bibliography{wide}

\bigskip

\noindent Department of Mathematics, University of Wisconsin, Madison, WI 53706
\newline \noindent  \texttt{rkent@math.wisc.edu}   

\bigskip

\noindent Department of Mathematics, University of Illinois, Urbana-Champaign, IL 61801
\newline \noindent \texttt{clein@math.uiuc.edu}
}

\end{document}